\documentclass[12pt]{amsart}

\usepackage{amssymb,mathrsfs,amsmath,amsthm,color,bm,mathtools,bbm,wasysym,float,subfig}
\usepackage[pagebackref]{hyperref}
\usepackage[noadjust]{cite}

\usepackage{enumerate}
\usepackage[centering]{geometry}
\theoremstyle{plain}
\newtheorem{thm}{Theorem}[section]
\newtheorem{defn}[thm]{Definition}

\newtheorem{prop}[thm]{Proposition}

\newtheorem{lem}[thm]{Lemma}

\theoremstyle{remark}
\newtheorem{rem}{Remark}
\numberwithin{equation}{section}

\DeclareMathOperator{\hdim}{\dim_H}
\DeclareMathOperator{\var}{Var}
\DeclareMathOperator{\hdimu}{\overline{\dim}_H}
\DeclareMathOperator{\hdiml}{\underline{\dim}_H}

\newcommand{\essinf}{\operatorname*{ess\,inf}}
\newcommand{\esssup}{\operatorname*{ess\,sup}}

\newcommand{\dif}{ \, \mathrm d}

\newcommand{\dist}{\mathrm{dist}}
\newcommand{\Q}{\mathbb Q}
\newcommand{\N}{\mathbb N}
\newcommand{\R}{\mathbb R}

\newcommand{\lm}{\mathcal L}
\renewcommand{\hm}{\mathcal H}

\newcommand{\hc}{\mathcal H_\infty}

\newcommand{\cb}{\mathcal B}

\newcommand{\cf}{\mathcal F}

\newcommand{\scg}{\mathscr G}

\newcommand{\qaq}{\mathrm{\quad and\quad}}
\newcommand{\be}{\bm\epsilon}

\begin{document}
		\title[A dimensional mass transference principle]{A dimensional mass transference principle from balls to open sets and applications to dynamical Diophantine approximation}
	\author{Yubin He}

	\address{Department of Mathematics, Shantou University, Shantou, Guangdong, 515063, China}

	\email{ybhe@stu.edu.cn}

%	\author{Lingmin Liao}
%
%	\address{School of Mathematics and Statistics, Wuhan University, Wuhan, Hubei 430072, China}
%
%	\email{lmliao@whu.edu.cn}


	\begin{abstract}
The mass transference principle of Beresnevich and Velani is a powerful mechanism for determining the Hausdorff dimension/measure of $\limsup$ sets that arise naturally in classical Diophantine approximation. However, in the setting of dynamical Diophantine approximation, this principle often fails to apply effectively, as the radii of the balls defining the dynamical $\limsup$ sets generally depend on the orbit of the point $x$ itself.

In this paper, we develop a dimensional mass transference principle that enables us to recover and extend classical results on shrinking target problems, particularly for the $\beta$-transformations and the Gauss map. Moreover, our result shows that the corresponding $\limsup$ sets have large intersection properties. A potentially interesting feature of our method is that, in many cases, shrinking target problems are closely related to finding an appropriate Gibbs measure, which may reveal new aspects of the link between thermodynamic formalism and dynamical Diophantine approximation.
	\end{abstract}
	\maketitle
\section{Introduction}\label{s:intro}
The central question in Diophantine approximation is: how well can a given
real number $x\in[0,1)$ be approximated by rational numbers. Dating back to Dirichlet, a consequence of his famous theorem is that for any $x\in[0,1)$,
\begin{equation}\label{eq:Dirichlet}
	\bigg|x-\frac{p}{q}\bigg|<\frac{1}{q^2}\quad\text{for i.m. $\frac{p}{q}\in\Q$},
\end{equation}
where i.m. stands for {\em infinitely many}.
The estimate above provides an approximation rate valid for all $x$ and lays the foundation for the metric theory of Diophantine approximation. This theory seeks to understand the sets of $x$ for which inequalities analogous to \eqref{eq:Dirichlet} hold, but with the right-hand-side replaced by functions of $q$ that decay more rapidly. For any $\tau\ge 2$, define
\[W(\tau):=\bigg\{x\in[0,1):\bigg|x-\frac{p}{q}\bigg|<\frac{1}{q^\tau}\text{ for i.m. $\frac{p}{q}\in\Q$}\bigg\}.\]
A classical result, proved independently by Besicovitch \cite{BesicovitchJLMSDiophantine} and Jarn\'ik \cite{Jarnikdimension}, shows that for any $\tau\ge2$,
\begin{equation}\label{eq:BesicovitchJarnik}
	\hdim W(\tau)=2/\tau,
\end{equation}
where $\hdim$ denotes the Hausdorff dimension. Remarkably, a profound connection between the statements described in \eqref{eq:Dirichlet} and \eqref{eq:BesicovitchJarnik} was uncovered by Beresnevich and Velani \cite{BeresnevichVelaniMTPann} through their celebrated mass transference principle, a powerful tool for deriving lower bounds on the Hausdorff dimension of a broad class of $\limsup$ sets. More specifically, Dirichlet theorem alone suffices to deduce the Besicovitch--Jarn\'ik theorem via their principle. We begin with some notation before stating their principle.

Throughout, the symbols
$\ll$ and $\gg$ will be used to indicate an inequality with an unspecified positive multiplicative constant. If $a\ll b$ and $a\gg b$, we write $a\asymp b$ and say that the quantities $a$ and $b$ are comparable.
Let $|E|$ denote the diameter of a set $E$. Let $X$ be a compact metric space equipped with a non-atomic
probability measure $\mu$. Suppose there exists a constant $\delta>0$ such that for any $x\in X$ and $0<r<|X|$
\[\mu(B(x,r))\asymp r^\delta,\]
where the implied constant does not depend on $x$ and $r$.
Such a measure is said to be {\em $\delta$-Ahlfors regular}. If $X$ supports a $\delta$-Ahlfors regular measure $\mu$, then we have that $\hdim X=\delta$ and $\mu$ is comparable to the $\delta$-dimensional Hausdorff measure (denoted by $\hm^\delta$) restricted to $X$ --- see \cite{Falconer_book} for the details.

The following statement is a simplified and slightly reformulated version of the result in \cite{BeresnevichVelaniMTPann}, adapted for our purposes.
\begin{thm}[Mass transference principle {\cite[Theorem 3]{BeresnevichVelaniMTPann}}]\label{t:MTPBV}
	 Let $X$ be a compact metric space equipped with a $\delta$-Ahlfors regular measure $\mu$. Let $\{B(x_n,r_n)\}$ be a sequence of balls in $X$ with $r_n\to 0$ as $n\to\infty$. Suppose that
	\[\mu\Big(\limsup_{n\to\infty}B(x_n,r_n)\Big)=1.\]
	Then, for any $\tau>1$,
	\[\hdim \Big(\limsup_{n\to\infty} B(x_n,r_n^\tau)\Big)\ge \frac{\delta}{\tau}.\]
\end{thm}
 The mass transference principle in this form concerns $\limsup$ sets defined by balls, which is sufficient for many classical applications. However, many naturally occurring $\limsup$ sets in Diophantine approximation are defined in terms of rectangles, neighborhoods of resonant sets, or more general open sets. To address such cases, various extensions of the mass transference principle have been developed, allowing for $\limsup$ sets defined by a wider range of shapes. We refer the reader to \cite{AllenBeresnevichhyperplaneCompos,AllenBakerMTPselect,KoivusaloRamsMTPIMRN,ZhongMTPJMAA,WangWuMTPrectangleMA,LiLiaoVelaniWangZorinMTPadv,HeMTPadv,DaviaudMTPJMAA,DaviaudMTPadv,ErikssonMTPadv,PerssonMTPLIPreal} for further details.

Classical Diophantine approximation concerns the distribution of rational approximations to real numbers. In recent years, this classical viewpoint has been naturally extended to the setting of dynamical Diophantine approximation, which studies approximation properties along orbits of dynamical systems. Among various problems in this field, our primary focus is on the shrinking target problem, first introduced by Hill and Velani \cite{HillVelaniSTPintroduce}, along with its generalizations, which concern whether the orbit of a given point hits a sequence of shrinking targets infinitely often.

Let $(X,d,T)$ be a dynamical system. The shrinking target problem studies the size, expressed in terms of dimension and measure, of the shrinking target set
\[\{x\in X:d(T^nx,x_0)<\psi(n,x)\text{ for i.m. $ n $}\},\]
where $x_0\in X$ and $\psi:\N\times X\to\R_{\ge 0}$ is a positive function. Numerous results on the measure and dimension of shrinking target sets have been established in various dynamical systems; see, for example, \cite{AllenBaranySTPMathematika,BakerKoivusaloSTPadv,BaranyRamsSTPPLMS,CoonsHussainWangSTPbeta,HeSTPETDS,HussainLiSimmonsWangBCETDS,KleinbockZhengBCNonlinearity,LiLiaoVelaniZorinSTPadv,Philipp1967,TanWangRecurbetaadv,WangSTPJMAA}.
To illustrate, consider the \emph{doubling map} $T_2(x) = 2x \ (\textrm{mod}\ 1)$ on $[0,1)$. In this setting, we are interested in the shrinking target set
\[
W(T_2,f,x_0) := \left\{ x \in [0,1): |T_2^n x -x_0|< e^{-S_n f(x)} \ \text{for i.m. } n \right\},
\]
where $x_0\in [0,1]$, $f: [0,1]\to \mathbb{R}_+$ is a H\"older continuous function and
\[
S_n f(x) = \sum_{k=0}^{n-1} f(T_2^k x)
\]
is the Birkhoff sum of $f$ along the orbit of $x$. It is already known (see e.g. \cite{WangWusurvey}) that
\[\hdim W(T_2,f,x_0)=s,\]
with $s$ satisfying
\begin{equation}\label{eq:P(-s(f+log 2),T2)=0}
	P(-s(f+\log 2), T_2):=\lim_{n\to\infty}\frac{1}{n}\log\sum_{i=0}^{2^n-1}\big(2^{-n}e^{-S_nf(y_{n,i})}\big)^s=0,
\end{equation}
where $y_{n,i}\in[i2^{-n},(i+1)2^{-n}]$ is chosen arbitrarily, and $P(\cdot,T_2)$ denotes the pressure function with respect to $T_2$ (see Section \ref{ss:presure function beta} for further details).

A natural question is whether the mass transference principle stated in Theorem \ref{t:MTPBV} can be applied to obtain the Hausdorff dimension of $W(T_2,f,x_0)$. To explore this, note that $f$ is H\"older continuous, and so $W(T_2,f,x_0)$ can be roughly expressed as
\[\bigcap_{N=1}^\infty\bigcup_{n=N}^\infty\bigcup_{i=0}^{2^n-1}B(x_{n,i},2^{-n}e^{-S_nf(x_{n,i})})=:\bigcap_{N=1}^\infty\bigcup_{n=N}^\infty E_n,\]
where $x_{n,i}\in[i2^{-n},(i+1)2^{-n}]$ satisfies $T_2^nx_{n,i}=x_0$.
If $f$ is not constant, then for a fixed $n \in \N$, the terms $e^{-S_n f(x_{n,i})}$ are not identical as $i$ varies. Therefore, $E_n$ consists of intervals of varying lengths, some of which may be very small, while others may be quite large (see Figure \ref{fig:E_n} for an illustration and Figure \ref{fig:E_np} for a comparison). Although the initial intervals are well-separated, enlarging them alters the situation, which constitutes the most significant difference compared to $W(\tau)$.
\begin{figure}[H]
	\centering
	\includegraphics{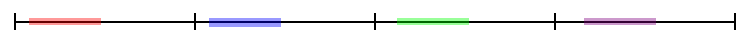}
	\caption{When $f$ is constant, $E_2$ consists of intervals of equal length.}
	\label{fig:E_np}
\end{figure}
\begin{figure}[H]
	\centering
	\includegraphics{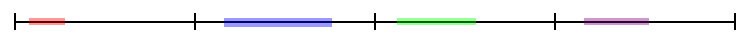}
		\caption{When $f$ is not constant, $E_2$ consists of intervals of varying lengths.}
		\label{fig:E_n}
\end{figure}
\noindent To be able to apply Theorem \ref{t:MTPBV}, it is necessary to enlarge the intervals in $E_n$'s to obtain a larger $\limsup$ set with full one-dimensional Lebesgue measure (denoted by $\lm$) restricted to $[0,1]$. It follows from \eqref{eq:P(-s(f+log 2),T2)=0} that for any $\epsilon>0$ and any sufficiently large $n$ (depending on $\epsilon$),
\[\sum_{i=0}^{2^n-1} \left( 2^{-n} e^{-S_n f(x_{n,i})} \right)^s = O \left( e^{n\epsilon} \right),\]
which means that the total length of the intervals $B\big(x_{n,i},\big(2^{-n}e^{-S_nf(x_{n,i})}\big)^s\big)$ in
\[E_{n,s}:=\bigcup_{i=0}^{2^n-1}B\big(x_{n,i},\big(2^{-n}e^{-S_nf(x_{n,i})}\big)^s\big)\]
is quite large.
However, this neither guarantees that
\begin{equation}\label{eq:not full}
	[0,1]=E_{n,s}\quad\text{nor that}\quad \lm\Big(\limsup_{n\to\infty}E_{n,s}\Big)=1,
\end{equation}
because the lengths of these intervals are different, making it possible for them to overlap (see Figure \ref{fig:E_ns} for an illustration and Figure \ref{fig:E_nsp} for a comparison).
\begin{figure}[H]
	\centering
	\includegraphics{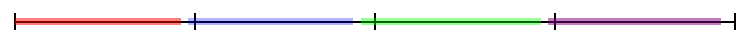}
	\caption{When $f$ is constant, the enlarged intervals contained in $E_{2,s}$ are well-separated and therefore efficiently cover $[0,1]$.}
	\label{fig:E_nsp}
\end{figure}
\begin{figure}[H]
	\centering
	\includegraphics{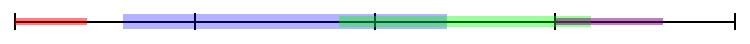}
	\caption{When $f$ is not constant, the enlarged intervals contained in $E_{2,s}$ may overlap and therefore fail to cover $[0,1]$ efficiently.}
	\label{fig:E_ns}
\end{figure}
Therefore, one should choose a parameter $t$ significantly smaller than $s$ to ensure that one of the equalities in \eqref{eq:not full} holds with $s$ replaced by $t$, i.e., $E_{n,t}$ effectively covers $[0,1]$.
This leads to the situation that although Theorem \ref{t:MTPBV} is still applicable, it does not directly provide the desired dimension estimates.

To address this shortfall, Wang and Zhang \cite{WangZhangDMPMathZ} developed an alternative mass transference principle from a dynamical perspective. Utilizing their principle, they successfully recovered the Hausdorff dimension of $W(T_2,f,x_0)$. Although the dimension result for $W(T_2,f,x_0)$ has been known, their work is significant in providing a new framework that connects shrinking target problems with mass transference principle. However, their principle does not extend to more general transformations such as the $\beta$-transformations or the Gauss map, nor can it be applied to settings where targets are defined by arbitrary open sets rather than balls.  The main goal of this paper is to address precisely this issue. Our purpose is to develop a framework capable of handling shrinking target problems --- along with various generalizations --- for the $\beta$-transformations and the Gauss map, and to extend the theory beyond the classical setting of balls to more general open sets. It is also worth noting that the setting in \cite{WangZhangDMPMathZ} is somewhat technical --- for example, it relies on the dynamical ubiquity and topological exactness assumptions --- and is almost entirely different from the current one that we will describe.

Part of the inspiration for our approach originates from the work of Barral and Seuret~\cite{BarralSeuretMTP} and Daviaud~\cite{DaviaudMTPJMAA}, who established that for a quasi-Bernoulli probability measure $\nu$ (see, e.g., \cite[Definition 2.3]{DaviaudMTPJMAA}),
\begin{equation}\label{eq:inhoMTP}
	\begin{split}
		&\nu\Big(\limsup_{n\to\infty}B(x_n,r_n)\Big)=1\quad\text{with\quad$\lim_{n\to\infty}r_n=0$}\\
		\Longrightarrow\quad&\hdim\Big(\limsup_{n\to\infty}B(x_n,r_n^\tau)\Big)\ge \frac{\hdim \nu}{\tau}\quad\text{ for any $\tau>1$}.
	\end{split}
\end{equation}
Here, the Hausdorff dimension of a measure $\nu$ is defined via its lower local dimension at $x$,
\begin{equation}\label{eq:definition of local dimension}
	\underline{D}(\nu,x):=\liminf_{r\to 0}\frac{\log\nu(B(x,r))}{\log r},
\end{equation}
and the lower and upper Hausdorff dimensions of $\nu$ are given by
\[\begin{split}
	\hdiml\nu:=&\essinf \underline{D}(\nu,x)=\inf\{\hdim E:E\text{ is a Borel set with }\nu(E)>0\},\\
	\hdimu\nu:=&\esssup \underline{D}(\nu,x)=\inf\{\hdim E:E\text{ is a Borel set with }\nu(E)=1\}.
\end{split}\]
If $\hdiml \nu = \hdimu \nu$, their common value is denoted by $\hdim \nu$. However, the quasi-Bernoulli property holds for the Gibbs measures associated with the doubling map, but generally fails for the $\beta$-transformations and the Gauss map. This limitation motivates the development of new concepts capable of handling these cases.  To this end, we introduce the notion of {\em quasi-self-conformality} of a measure.
\begin{defn}[Quasi-self-conformality]\label{d:quasiselfconf}
	Let $\nu$ be a Borel probability measure supported on a metric space $X$, and let $\mathcal{F} = \{F_n\}$ be a collection of closed subsets of $X$. We say that $\nu$ is quasi-self-conformal with respect to $\mathcal{F}$ if there exists a constant $C \ge 1$ such that for every $F_n \in \mathcal{F}$, there exists a bijection $f_n: F_n \to X$ satisfying:
	\begin{enumerate}[\upshape (1)]
		\item $C^{-1} \dfrac{|x - y|}{|F_n|} \le |f_n(x) - f_n(y)| \le C \dfrac{|x - y|}{|F_n|}$ for all $x, y \in F_n$;
		\item The normalized pushforward measure $\nu^{(n)} := \dfrac{\nu \circ f_n^{-1}}{\nu(F_n)}$ satisfies
		\[
		C^{-1} \nu(E) \le \nu^{(n)}(E) \le C\nu(E)\quad\text{for any Borel set $E$}.
		\]
	\end{enumerate}
\end{defn}
The notion of quasi-self-conformality arises as an appropriate generalization of the classical concept of self-conformality for sets, designed to capture approximately self-conformal structures exhibited by measures.

To formulate our main result, we recall the notion of Hausdorff content. In this paper, we focus on the case where the ambient space $X$ is a compact subset of $\mathbb{R}^d$. For any $s \ge 0$ and a set $E\subset \R^d$, the \emph{$s$-dimensional Hausdorff content} of $E$ is defined by
\[\hc^s(E)=\inf\bigg\{\sum_{i}|B_i|^s:E\subset \bigcup_{i\ge 1}B_i, \text{ where $B_i$ are balls}\bigg\}.\]
Our method further enables us to establish the so-called \emph{large intersection property}, introduced and systematically studied by Falconer \cite{FalconerLIPintroduce}.
\begin{defn}[\cite{FalconerLIPintroduce}]\label{d:LIP}
	Let $0<s\le \hdim X$. We define $\scg^s(X)$ to be the class of $G_\delta$-subsets $E$ of $X$ such that there exists a constant $c>0$ such that for any $0<t<s$ and any ball $B$,
	\begin{equation}\label{eq:defhcbound}
		\hc^t(E\cap B)>c\hc^t(B).
	\end{equation}
\end{defn}
Falconer originally defined this property in $\mathbb{R}^d$, and it has later been extended to general metric spaces. It has been shown, for example in \cite[Theorem 2.4]{HeMTPadv}, that if $X$ supports a $\delta$-Ahlfors regular measure, then the class $\scg^s(X)$ is closed under countable intersections, and moreover,
\[\hdim E\ge s\quad\text{for all $E\in \scg^s(X)$}.\]

\begin{thm}\label{t:MTP}
	Let $X\subset\R^d$ be a compact subset equipped with a $\delta$-Ahlfors regular probability measure $\mu$. Let $\nu$ be a quasi-self-conformal probability measure with respect to a collection of closed sets $\mathcal{F} = \{F_n\}$ in $X$, such that
	\begin{equation}\label{eq:fullambient}
		\mu\Big(\limsup_{n\to\infty} F_n\Big) = 1\qaq\lim_{n\to\infty}|F_n|=0.
	\end{equation}
	Suppose that there exist a sequence of balls $\{B(x_n,r_n)\}$ with centers in $X$, and a sequence of open sets $\{E_n\}$ in $X$, satisfying the following conditions:
	\begin{enumerate}
		\item $r_n\to 0$ as $n\to\infty$;
		\item $\nu\big(\limsup B(x_n,r_n)\big)=1$;
		\item $ E_n\subset B_n $. Moreover, there exists a constant $s\ge 0$ such that
		\[\hc^s(E_n)\gg r_n^{\hdimu\nu},\]
		where the implied constant is independent of $n$.
	\end{enumerate}
	Then,
	\[\limsup_{n\to\infty}E_n\in \scg^s(X).\]
\end{thm}

\begin{rem}
The assumption that $X \subset \mathbb{R}^d$ is essential, as our arguments rely on the Besicovitch covering theorem, which generally does not hold in  metric spaces. This arises because the measure $\nu$ is generally not doubling, which prevents many standard covering lemmas from applying effectively --- except for the Besicovitch covering theorem.
\end{rem}
\begin{rem}
	The condition in \eqref{eq:fullambient} serves to characterize the extent to which the measure $\nu$ exhibits quasi-self-conformality. In many case, a measure $\nu$ satisfy this condition is a Gibbs measure, and thus generally singular to the ambient measure $\mu$. For example, if $\mu$ is the Lebesgue measure, and $\nu$ is the $(p,1-p)$-Bernoulli measure with $0<p<1$, then we can take $F_n$ to be the dyadic intervals and $f_n$ be the linear map that sends $F_n$ onto $[0,1]$. Therefore, $\nu$ satisfies \eqref{eq:fullambient}. We will discuss additional examples of interest, such as the $\beta$-transformations in Section \ref{s:beta} and the Gauss map in Section \ref{s:Gauss}.
\end{rem}
\begin{rem}
It is important to highlight the differences and the novelty of Theorem \ref{t:MTP} in comparison with existing `balls to open sets' mass transference principles \cite{HeMTPadv,KoivusaloRamsMTPIMRN,PerssonMTPLIPreal,ZhongMTPJMAA}. First, the results of the latter focus on the cases $\nu=\mu$, in which the estimate $\mu(B(x,r))\asymp r^\delta$ for arbitrary balls plays a crucial role in the proof. However, if $\nu$ is singular with respect to $\mu$, then $\nu$ does not satisfy this estimate and, worse still, may fail to be doubling. Therefore, new approaches are required, and this is one of the reasons why we introduce the notion of quasi-self-conformality; for further details, see Remark \ref{r:not Ahlfors}. Second, as illustrated by the example of $W(T_2,f,x_0)$, the results in \cite{HeMTPadv,KoivusaloRamsMTPIMRN,PerssonMTPLIPreal,ZhongMTPJMAA} cannot be applied in a straightforward manner to determine the Hausdorff dimension of $W(T_2,f,x_0)$. As we will see in Section \ref{ss:proofbeta}, together with some known results from thermodynamic formalism, Theorem \ref{t:MTP} can be directly applied to obtain the lower bound for $\hdim W(T_2,f,x_0)$, in the same manner that Theorem \ref{t:MTPBV} is used to obtain $\hdim W(\tau)$. This largely demonstrates the effectiveness of our theorem in addressing Diophantine approximation problems in dynamical systems.
\end{rem}

We now apply Theorem \ref{t:MTP} to recover and extend classical results for the $\beta$-transformations and the Gauss map. For a Lipschitz function $h:X\to X$, its {\em Lipschitz constant} is defined as the smallest $L>0$ such that for any $x,y\in X$,
\[|h(x)-h(y)|\le L|x-y|.\]
Let $\{h_n\}$ be a sequence of Lipschitz functions with uniformly bounded Lipschitz constants and let $f$ be a positive and continuous function defined on $X$. The modified shrinking target sets are defined as
\[W(T,f,\{h_n\})=\{x\in X: |T^nx-h_n(x)|<e^{-S_nf(x)}\text{ for i.m. $n$}\}.\]
\begin{thm}\label{t:recoverspt}
	Let $f$ be a positive and continuous function defined on $X$. Suppose that $X=[0,1]$ and that $T$ is either the $\beta$-transformation or the Gauss map. Then,
	\[W(T,f,\{h_n\})\in\scg^s([0,1]),\]
	where $s$ is the unique solution to $P(-s(f+\log|T'|),T)=0$.
\end{thm}
\begin{rem}
	The lower bound for the Hausdorff dimension of $W(T,f,\{h_n\})$ implied in Theorem \ref{t:recoverspt} was previously established in \cite{BugeaudWangSTPJFG,TanWangRecurbetaadv,LiWangWuXuSTPPLMS,WangSTPJMAA}. However, those results rely on the construction of large Cantor-type sets and do not imply the large intersection property. Interestingly, Theorem \ref{t:MTP} offers a different perspective: the problem is reduced to seeking a suitable Gibbs measure and estimating its Hausdorff dimension. This perspective may provide new insights into the interplay between thermodynamic formalism and dynamical Diophantine approximation. It can be seen from our proof that Theorem \ref{t:MTP} is also applicable to the shrinking target problems in expanding Markov maps with finite partitions on $[0,1]$, see Remark \ref{r:applicable to Markov} for further details. The only reason we restrict ourselves to these two cases is that the
	$\beta$-transformations and the Gauss map are among the most well-known non-Markov map and expanding map with infinitely many branches, respectively.
\end{rem}
\begin{rem}
	After completing the proofs of our main results, the author became aware that Daviaud \cite{DaviaudSTPETDS} had employed some similar ideas to study the shrinking target problem for self-conformal sets with overlaps. However, his results neither imply the large intersection property nor can they be directly applied to the $\beta$-transformations or the Gauss map.
\end{rem}
Theorem~\ref{t:recoverspt} is a direct application of Theorem~\ref{t:MTP}, where the sets $E_n$ are taken to be balls. To further demonstrate the versatility of our main result, we present a concise proof of the following theorem, which was also previously established in~\cite{HuangWuXuconsecutiveISM}.
Let $m\ge 1$ be an integer and $B>1$. Define
\[F_m(B):=\{x\in[0,1): a_{n+1}(x)\cdots a_{n+m}(x)\ge B^n\text{ for i.m. $n$}\},\]
where $a_n(x)$ denotes the $n$th partial quotient of $x$ (see Section~\ref{s:Gauss} for the definition).
\begin{thm}\label{t:consecutive}
	Let $m\ge 1$ be an integer and $B>1$. Then,
	\[F_m(B)\in\scg^u([0,1]),\]
for some $u \in (1/2, 1)$ satisfying \[P(-u \log|G'| - g_m(u) \log B, G) = 0,\]
where the function $g_m(u)$ is given by
\[g_m(u)=\frac{u^m(2u-1)}{u^m-(1-u)^m}.\]
\end{thm}
The structure of the paper is as follows. In Section~\ref{s:pre}, we collect several foundational results and technical tools that will be used throughout the paper. Section~\ref{s:MTP} is devoted to the proof of our main result, i.e. Theorem \ref{t:MTP}. In Section~\ref{s:beta}, we review the definition and key properties of the $\beta$-transformations, and apply our main result to obtain the large intersection property of $W(T_\beta,f,\{h_n\})$ in this setting. Section~\ref{s:Gauss} serves a similar purpose for the Gauss map: we recall its basic properties and then apply our theorem to derive the large intersection properties of $W(G,f,\{h_n\})$ and $F_m(B)$.

\section{Preliminary}\label{s:pre}
This section recalls key tools from geometric measure theory and covering arguments that underpin the main results of this paper. We start by presenting two fundamental results that relate the Hausdorff content of a set to probability measures exhibiting appropriate local dimension estimates. Here and hereafter, we will assume that $X\subset \R^d$ is compact equipped with a $\delta$-Ahlfors regular measure $\mu$.
\begin{prop}[Mass distribution principle {\cite[Lemma 1.2.8]{BishopPeresbook}}]\label{p:MDP}
	Let $ E $ be a Borel subset of $ \R^d $. Suppose that $ E $ supports a Borel probability measure $ \lambda $ that satisfies
	\[\lambda(B(x,r))\le cr^s,\]
	for all $x\in\R^d$ and $r>0$, where $ 0<c<\infty $ is a constant. Then
	\[\hc^s(E)\ge1/c .\]
\end{prop}
\begin{lem}[{Frostman's lemma \cite[Theorem 8.8]{Mattilageometry1999}}]\label{l:frostman}
		Let $ E $ be a Borel subset of $ \R^d $. If $\hc^s(E)>0$, then there exists a probability measure $\lambda$ supported on $E$ such that for any $x\in\R^d$ and $r>0$,
	\[\lambda(B(x,r))\ll \frac{r^s}{\hc^s(E)},\]
	where the unspecified constant depends only on $d$.
\end{lem}
Theorem~\ref{t:lip} below offers a relatively simple criterion for verifying that a $\limsup$ set has the large intersection property.
\begin{thm}[{\cite[Corollary 2.6]{HeMTPadv}}]\label{t:lip}
	Let $0< s\le\hdim X$. Let $\{E_n\}$ be a sequence of open sets in $X$. If for any $0<t<s$, there exists a constant $c=c(t)>0$ such that
	\[\limsup_{n\to\infty} \hc^t (E_n\cap  B)>c\mu(B)\asymp |B|^\delta\]
	holds for any ball $B\subset X$, then
	\[\limsup_{n\to\infty} E_n\in\scg^s(X).\]
\end{thm}
\begin{rem}
	Although the Falconer's original definition of large intersection property (see Definition \ref{d:LIP}) requires that
	\[\hc^t (E\cap  B)\gg \hc^t (B)\asymp|B|^t\quad\text{holds for any ball $B\subset X$},\]
it was observed by the author in \cite{HeMTPadv} that this condition can be significantly weakened, as stated in the above theorem. It is also worth noting that several alternative definitions of large intersection properties appear in that paper, some of which may be used to derive Hausdorff measures. However, we will not pursue these here, since our focus is on Hausdorff dimension.
\end{rem}
The following covering result, due to Besicovitch, will be used to efficiently select disjoint subfamilies of balls covering a given set.
\begin{thm}[{Besicovitch covering Theorem \cite[Theorem 2.7]{Mattilageometry1999}}] There is a positive integer $Q_d$ depending only on the dimension $d$ with the following property. Let $E\subset \R^d$ be a bounded set, and let $\cb$ be a family of balls such that each point of $E$ is the centre of some ball of $\cb$. There are families $\cb_1,\dots,\cb_{Q_d}\subset\cb$ covering $E$ such that each $\cb_k$ is disjoint, that is,
	\[E\subset \bigcup_{1\le k\le Q_d}\bigcup_{B\in\cb_k}B\]
	and
	\[B\cap B'=\emptyset\quad\text{for $B,B'\in\cb_k$ with $B\ne B'$}.\]
\end{thm}
The next lemma allows us to extract well-separated subcollections from a sequence of shrinking balls while retaining a definite portion of total measure.
\begin{lem}[{\cite[Lemma 5]{AllenBakerMTPselect}}]\label{l:kgb}
	Let $\{B(x_n,r_n)\}$ be a sequence of balls in $X\subset\R^d$ such that
	\[\mu\bigg(\limsup_{n\to\infty}B(x_n,r_n)\bigg)=1 \qaq \text{$r_n\to 0$ as $n\to\infty$.}\]
	Then, for any ball $B$ in $X$, there exists a finite collection
	\[K_{B}\subset\{B(x_n,r_n)\}\]
	satisfying the following properties:
	\begin{enumerate}[\upshape(1)]
		\item $B(x_n,r_n)\subset B$ for all $B(x_n,r_n)\in K_B$;
		\item If $B(x_n,r_n), B(x_m,r_m)\in K_B$ are distinct, then $B(x_n,3r_n)\cap B(x_m,3r_m)=\emptyset$;
		\item $\mu\big(\bigcup_{B(x_n,r_n)\in K_B}B(x_n,r_n)\big)\gg\mu(B),$ where the implied constant does not depend on  $B$.
	\end{enumerate}
\end{lem}
Here, we highlight the difference between Besicovitch covering theorem and Lemma \ref{l:kgb}. Besicovitch covering theorem applies to arbitrary Borel measures but is restricted to Euclidean spaces, whereas Lemma \ref{l:kgb} can be extended to general metric spaces but requires the measure to be doubling. In the sequel, when it is necessary to extract a disjoint subcollection from a sequence of shrinking balls, we will use Lemma \ref{l:kgb} for the ambient measure $\mu$, and the Besicovitch covering theorem for the reference measure $\nu$.

Note that as the collection $\{F_n\}$ of closed sets in Theorem \ref{t:MTP} are not necessarily balls, the following variant of Lemma \ref{l:kgb} is needed. Before moving on, we give some simple properties of the map $f_n$ and the set $F_n$ defined in Definition \ref{d:quasiselfconf} that will be used soon.
\begin{prop}\label{p:property of conformal map f}
		Let $F$ be a closed subset of $X$ and $f:F\to X$ be a bijection. Suppose that there exists a constant $C \ge 1$ such that
		\begin{equation}\label{eq:conformal f}
			C^{-1} \dfrac{|x - y|}{|F|} \le |f(x) - f(y)| \le C \dfrac{|x - y|}{|F|}\quad \text{for all $x, y \in F$},
		\end{equation}
	then the following properties hold:
	\begin{enumerate}[\upshape(1)]
		\item For any $x\in F$ and $0<r<|F|$,
		\[B(f(x),C^{-1}r/|F|)\subset f(B(x,r)\cap F)\subset B(f(x),C r/|F|).\]
		\item there exists a constant $c>0$ depends on $X$, $C$ and $\mu$ only such that
		\[F\subset B(x,c|F|)\qaq c^{-1}|F|^\delta\le \mu(F)\le c|F|^\delta,\]
		where $x\in F$.
	\end{enumerate}
\end{prop}
\begin{proof}
	(1) Let $x\in F$ and $0<r<|F|$. For any $y\in B(x,r)\cap F$, we have
	\[|f(x)-f(y)|\stackrel{\eqref{eq:conformal f}}{\le}Cr/|F|,\]
	and therefore $f(B(x,r)\cap F)\subset B(f(x),C r/|F|)$.

	For the left inclusion in item (1), take an arbitrary point $z$ in $B(f(x),C^{-1}r/|F|)$. Since $f$ is a bijection, there exists $y\in F$ such that $f(y)=z$. It follows that
	\[|x-y|\stackrel{\eqref{eq:conformal f}}{\le} C|F|\cdot|f(x)-f(y)|=C|F|\cdot|f(x)-z|<r.\]
	Therefore, $y\in B(x,r)\cap F$, and consequently, $z=f(y)\in f(B(x,r)\cap F)$.

	(2) Suppose, for the sake of contradiction, that $	F \nsubseteq B(x, 2C\,|X|\cdot|F|).$
	Then, for any $y \in F \setminus B(x, 2C|X|\cdot|F|)$, we would have
	\[
	|f(x) - f(y)| \stackrel{\eqref{eq:conformal f}}{\ge} C^{-1} \frac{|x - y|}{|F|} \ge 2|X|,
	\]
	which is a contradiction. Hence,
	\[
	F \subset B(x, 2C\,|X|\cdot|F|).
	\]
	It then follows that
	\[
	\mu(F) \le \mu\big(B(x, 2C\,|X|\cdot|F|)\big) \ll |F|^\delta,
	\]
	where the implied constant depends only on \(X\), \(C\), and \(\mu\).

	Next, the second inequality in \eqref{eq:conformal f} implies that $f$ is Lipschitz with Lipschitz constant at most $C|F|^{-1}$. Hence,
	\[1=\mu(X)\asymp\hm^\delta(X)=\hm^\delta(f(F))\le (C|F|^{-1})^\delta\hm^\delta(F),\]
	where the last inequality follows from \cite[Proposition 3.1]{Falconer_book}. Since the $\delta$-Ahlfors regular measure $\mu$ is comparable to $\hm^s|_X$, we have
	\[\mu(F)\asymp\hm^\delta|_X(F)=\hm^\delta(F)\ge C^{-\delta}|F|^\delta,\]
	as desired.
\end{proof}

Proposition \ref{p:property of conformal map f} together with Lemma \ref{l:kgb} enables us to obtain the following result.
\begin{lem}\label{l:kgbf}
	Let $\cf=\{F_n\}$ be a sequence of closed sets in $X$ such that the following properties hold:
	\begin{enumerate}[\upshape(a)]
		\item there exists a constant $C \ge 1$ such that for every $F_n \in \mathcal{F}$, there exists a bijection $f_n: F_n \to X$ satisfying
		\[C^{-1} \dfrac{|x - y|}{|F_n|} \le |f_n(x) - f_n(y)| \le C \dfrac{|x - y|}{|F_n|} \text{\quad for all } x, y \in F_n;\]
		\item $\mu\Big(\limsup\limits_{n\to\infty}  F_n\Big)=1$ and $\lim\limits_{n\to\infty}|F_n|=0.$
	\end{enumerate}
	Then, for any ball $B$ in $X$, there exists a finite collection
	\[\cf_B\subset\cf\]
	satisfying the following properties:
	\begin{enumerate}[\upshape(1)]
		\item $F_n\subset B$ for all $F_n\in \cf_B$;
		\item If $F_n$ and $F_m$ are distinct, then $\dist(F_n,F_m)\ge\max\{|F_n|,|F_m|\}$, where the distance between sets is defined by
		\[
		\dist(F_n, F_m) := \inf\{ d(x,y) : x \in F_n,\ y \in F_m \};
		\]
		\item $\mu\big(\bigcup_{F_n\in \cf_B}F_n\big)\gg\mu(B),$ where the implied constant does not depend on $B$.
	\end{enumerate}
\end{lem}
\begin{proof}
	If follows from Proposition \ref{p:property of conformal map f} that item (a) of the lemma implies that
	\begin{equation}\label{eq:fnsize}
		F_n\subset B(x_n,c|F_n|)\qaq c^{-1}|F_n|^\delta\le \mu(F_n)\le c|F_n|^\delta,
	\end{equation}
	where $x_n\in F_n$ and $c\ge 1$ is a constant depend only on $X$, $C$ and $\mu$. By the full measure assumption in item (b), it is clear that
	\[\mu\Big(\limsup_{n\to\infty} B(x_n,c|F_n|)\Big)=1.\]
	This together with another assmuption $\lim_{n\to\infty}|F_n|=0$ in item (b), we have that Lemma \ref{l:kgb} is applicable to the sequence of balls $\{B(x_n,c|F_n|)\}$. By that lemma, we obtain a finite collection $K_B$ satisfying items (1)--(3) in Lemma \ref{l:kgb}. Let
	\[\cf_B=\{F_n: B(x_n,c|F_n|)\in K_B\}.\]
	By the first inclusion of \eqref{eq:fnsize}, item (1) in the lemma follows immediately from Lemma \ref{l:kgb} (1) since
	\[F_n\subset B(x_n,c|F_n|)\subset B.\]

	For item (2), suppose that $F_n,F_m\in \cf_B$ are distinct. Then, by definition, the same is true for $B(x_n,c|F_n|), B(x_m,c|F_m|)\in K_B$. It follows from Lemma \ref{l:kgb} (2) that
	\[B(x_n,3c|F_n|)\cap B(x_m,3c|F_m|)=\emptyset.\]
	Therefore,
	\[\dist\big(B(x_n,c|F_n|), B(x_m,c|F_m|)\big)\ge \max\{c|F_n|,c|F_m|\}\ge \max\{|F_n|,|F_m|\}. \]
	Consequently,
	\[\dist(F_n,F_m)\ge \dist\big(B(x_n,c|F_n|), B(x_m,c|F_m|)\big)\ge \max\{|F_n|,|F_m|\}.\]

	For item (3), since $F_n$ are pairwise disjoint, we have
	\[\begin{split}
		\mu\bigg(\bigcup_{F_n\in \cf_B}F_n\bigg)&=\sum_{F_n\in\cf_B}\mu(F_n)\stackrel{\eqref{eq:fnsize}}{\asymp} \sum_{F_n\in\cf_B}|F_n|^\delta\asymp \sum_{F_n\in\cf_B}\mu\big(B(x_n,c|F_n|)\big)\\
		&=\mu\bigg(\bigcup_{B(x_n,c|F_n|)\in K_B}B(x_n,c|F_n|)\bigg)\stackrel{\text{Lemma \ref{l:kgb} (3)}}{\gg}\mu(B).\qedhere
	\end{split}\]
\end{proof}
\section{Proof of Theorem \ref{t:MTP}}\label{s:MTP}
Let $\{E_n\}$ be as given in Theorem \ref{t:MTP}. In this section, our goal is to establish that $\limsup E_n$ has the large intersection property. Specifically, we will show that for any $0 < t < s$,
\begin{equation}\label{eq:content bound}
	\limsup_{\ell\to\infty}\hc^t\bigg(\bigcup_{k= \ell}^\infty E_k\cap B\bigg)\gg \mu(B) \quad\text{holds for any ball $B\subset X$},
\end{equation}
where the implied constant is independent of the ball $B$.
Once this is established, Theorem~\ref{t:lip} yields
\[\limsup_{\ell\to\infty}\bigg(\bigcup_{k= \ell}^\infty E_k\bigg)=\limsup_{n\to\infty} E_n\in\scg^s(X),\]
thereby completing the proof of the large intersection property for the set $\limsup E_n$.
\begin{rem}
	Although our proof relies on the original `balls to open sets' mass transference principle (that is, Theorem \ref{t:lip} or Theorem \ref{t:MTP} in the case $\mu=\nu$), it does not immediately imply Theorem \ref{t:MTP} itself, as can be seen from our proof of that theorem. The reason we employ the original `balls to open sets' mass transference principle is that, as observed in our previous paper \cite{HeMTPadv}, this principle has the potential to serve as a general framework for determining the Hausdorff measure/dimension of $\limsup$ sets in a relatively simple manner. Quite surprisingly, this principle continues to be applicable in the present new setting.
\end{rem}

To proceed, fix $0 < t < s$, $\ell \ge 1$, and a ball $B \subset X$. The remainder of this section is devoted to establishing the lower bound
\begin{equation}\label{eq:goal}
	\hc^t\bigg(\bigcup_{k= \ell}^\infty E_k\cap B\bigg)\gg \mu(B).
\end{equation}
%The rest of this section is arranged as follows. For any given $n\ge 1$, we first construct a subset of $\bigcup_{k=n}^\infty E_k\cap B$. Then, we distribute a probability measure $\eta$ supported on this subset and subsequently
\subsection{Construction of a subset of  $\bigcup_{k= \ell}^\infty E_k\cap B$} Our approach to constructing the desired subset is motivated by \cite{DaviaudMTPJMAA}, but the overall strategy we adopt to establish the lower bound of the Hausdorff dimension of $\limsup E_n$ is different.

Let $\nu$ be the reference measure stated in Theorem \ref{t:MTP}. Let $\varepsilon=s-t>0$. To make effective use of the local behavior of the measure $\nu$, we consider the set of points where the lower local dimension exceeds a certain threshold. By the definition of the lower Hausdorff dimension $\hdimu\nu$, the set
\[E_{\nu}^\varepsilon:=\{x\in X:\underline{D}(\nu,x)>\hdimu\nu-\varepsilon\}\]
has positive $\nu$-measure. Let us denote $\gamma_\varepsilon := \nu(E_\nu^\varepsilon)$ for convenience. To obtain a uniform estimate on the measure of small balls, we consider the sets
\[E_{\nu}^{n,\varepsilon}:=\{x\in X: \text{$\forall$ $0<r<1/n$, $\nu(B(x,r))\le r^{\hdimu\nu-\varepsilon}$}\}.\]
By the definition of $\underline{D}(\nu,x)$ (see \eqref{eq:definition of local dimension}), we have
\[E_{\nu}^\varepsilon=\bigcup_{n= 1}^\infty E_{\nu}^{n,\varepsilon},\]
and clearly the sequence $\{E_{\nu}^{n,\varepsilon}\}$ is increasing in $n$. Therefore, by the continuity of measure from below, there exists an integer $N = N(\varepsilon)$ such that
\begin{equation}\label{eq:gammaeps}
	\nu(E_{\nu}^{N,\varepsilon})\ge \gamma_\varepsilon/2.
\end{equation}

For the given ball $B$, let $\cf_B\subset \cf$ be the finite subcollection of closed sets obtained from Lemma \ref{l:kgbf}.
Recall from Definition \ref{d:quasiselfconf} that for any $F_i\subset\cf$,
\begin{equation}\label{eq:quasimu}
	\nu/C\le \nu^{(i)}=\frac{\nu\circ f_i^{-1}}{\nu(F_i)}\le C\nu,
\end{equation}
where $f_i:F_i\to X$ is a bijection.
Let $F_i\in\cf_B$. Since $f_i$ is a bijection, we can estimate the measure of the intersection $B(x,\rho)\cap F_i$ for $x\in F_i$ and $0<\rho<|F_i|$ as follows:
\begin{align}
	\nu(B(x,\rho)\cap F_i)&=\nu\big(f_i^{-1}(f_i(B(x,\rho)\cap F_i))\big)=\nu(F_i)\nu^{(i)}\big(f_i(B(x,\rho)\cap F_i)\big)\notag\\
	&\stackrel{\eqref{eq:quasimu}}{\le} C\nu(F_i)\nu\big(f_i(B(x,\rho)\cap F_i)\big)\notag\\
	&\stackrel{\text{Proposition \ref{p:property of conformal map f} (1)}}{\le} C\nu(F_i)\nu\big(B(f_i(x),C\rho/|F_i|)\big).\label{eq:quasiupper}
\end{align}

Let $x\in f_i^{-1}(E_\nu^{N,\varepsilon})$. Then, $x\in F_i$ and $f_i(x)\in E_\nu^{N,\varepsilon}$. By the  definition of $E_\nu^{N,\varepsilon}$, for any $0<r<1/N$,
\begin{equation}\label{eq:holder}
	\nu(B(f_i(x),r))\le r^{\hdimu\nu-\varepsilon}.
\end{equation}
Then, for any $0<\rho<|F_i|/(CN) $ (or equivalently $0<C \rho/|F_i|<1/N$),
\[\begin{split}
	\nu(B(x, \rho)\cap F_i)&\stackrel{\eqref{eq:quasiupper}}{\le} C\nu(F_i)\nu\bigg(B\bigg(f_i(x),\frac{C \rho}{|F_i|}\bigg)\bigg)\stackrel{\eqref{eq:holder}}{\le} C\nu(F_i)\cdot \bigg(\frac{C \rho}{|F_i|}\bigg)^{\hdimu\nu-\varepsilon}\\
	&\le C^{d+1}\nu(F_i)\cdot \bigg(\frac{\rho}{|F_i|}\bigg)^{\hdimu\nu-\varepsilon}.
\end{split}\]
Equivalently, for any such $\rho$,
\[\frac{\nu(B(x,\rho)\cap F_i)}{\nu(F_i)}\le C^{d+1}\bigg(\frac{\rho}{|F_i|}\bigg)^{\hdimu\nu-\varepsilon}.\]
Therefore, we conclude that
\[\begin{split}
	&f_i^{-1}(E_\nu^{N,\varepsilon})\\
	=&f_i^{-1}\big(\{x\in X: \text{$\forall$ $0<r<1/N$, $\nu(B(x,r))\le r^{\hdimu\nu-\varepsilon}$}\}\big)\\
	\subset&\bigg\{x\in F_i: \text{$\forall$ $0<\rho<\frac{|F_i|}{CN}$, $\frac{\nu(B(x,\rho)\cap F_i)}{\nu(F_i)}\le C^{d+1}\bigg(\frac{\rho}{|F_i|}\bigg)^{\hdimu\nu-\varepsilon}$}\bigg\}.
\end{split}\]
Define
\[\begin{split}
	E_{\nu,F_i}^{N,\varepsilon}:=&\limsup_{n\to\infty}B(x_n,r_n)\cap \\
	&\bigg\{x\in F_i: \text{$\forall$ $0<\rho<\frac{|F_i|}{CN}$, $\frac{\nu(B(x,\rho)\cap F_i)}{\nu(F_i)}\le C^{d+1}\bigg(\frac{\rho}{|F_i|}\bigg)^{\hdimu\nu-\varepsilon}$}\bigg\}.
\end{split}\]
Since $\nu(\limsup B(x_n,r_n))=1$ (by Theorem \ref{t:MTP} (2)), we have $f_i^{-1}(E_\nu^{N,\varepsilon})= E_{\nu,F_i}^{N,\varepsilon}$ except for a set of zero $\nu$-measure. Therefore,
\begin{equation}\label{eq:EvFi}
	\nu(E_{\nu,F_i}^{N,\varepsilon})=\nu\big(f_i^{-1}(E_\nu^{N,\varepsilon})\big)=\nu(F_i)\nu^{(i)}(E_\nu^{N,\varepsilon})\stackrel{\eqref{eq:quasimu}}{\ge} \frac{\nu(F_i)\nu(E_\nu^{N,\varepsilon})}{C}\stackrel{\eqref{eq:gammaeps}}{\ge} \frac{\gamma_\varepsilon\nu(F_i)}{2C}.
\end{equation}

For any $z\in E_{\nu,F_i}^{N,\varepsilon}$, there exists infinitely many $n$ such that $z\in B(x_n,r_n)$. Choose an integer $n_z\ge \ell$ large enough so that
\begin{equation}\label{eq:condnz}
	z\in B(x_{n_z},r_{n_z})\subset B\qaq16r_{n_z}\le |F_i|/(C N).
\end{equation}
The above inclusion $B(x_{n_z},r_{n_z})\subset B$ is possible since $E_{\nu,F_i}^{N,\varepsilon}\subset F_i\subset B$ and $B$ is open.
Set $L_{n_z}:=B(z,5r_{n_z})$. Recall from Theorem \ref{t:MTP} that $E_n\subset B_n$. Then, we have
\begin{equation}\label{eq:Enz}
	E_{n_z}\subset B(x_{n_z},r_{n_z})\subset L_{n_z}.
\end{equation}
Thus, the collection of balls $\{L_{n_z}:z\in E_{\nu,F_i}^{N,\varepsilon}\}$ forms a covering of $E_{\nu,F_i}^{N,\varepsilon}$. By Besicovitch covering theorem, one can extract from this cover a finite number (at most $Q_d$) of disjoint subcollections $\cb_k(F_i)$ for $1\le k\le Q_d$, such that:
\begin{enumerate}[(1)]
	\item Each collection $ \cb_k(F_i) $ consists of pairwise disjoint balls: for any distinct $ L_{n_z}, L_{n_w} \in \cb_k(F_i) $, it holds that $ L_{n_z} \cap L_{n_w} = \emptyset $;
	\item The union of these collections covers the entire set:
	\[
	E_{\nu,F_i}^{N,\varepsilon} \subset \bigcup_{1 \le k \le Q_d} \bigcup_{L_{n_z} \in \cb_k(F_i)} L_{n_z}.
	\]
\end{enumerate}
By item (2) above, there exists some $ 1 \le k_i \le Q_d $ such that the corresponding collection $ \cb_{k_i}(F_i) $ satisfies
\[\nu\bigg(\bigcup_{L_{n_z}\in\cb_{k_i}(F_i)}L_{n_z}\bigg)\ge \frac{\nu(E_{\nu,F_i}^{N,\varepsilon})}{Q_d}\stackrel{\eqref{eq:EvFi}}{\ge}\frac{\gamma_\varepsilon\nu(F_i)}{2CQ_d}.\]
Since the balls in $ \cb_{k_i}(F_i) $ are pairwise disjoint (see item (1) above), we may further extract a finite subcollection $ \cb(F_i) \subset \cb_{k_i}(F_i) $ such that
\begin{equation}\label{eq:mu>}
	\nu\bigg(\bigcup_{L_{n_z}\in\cb(F_i)}L_{n_z}\bigg)\ge \nu\bigg(\bigcup_{L_{n_z}\in\cb_{k_i}(F_i)}L_{n_z}\bigg)\bigg/2\ge \frac{\gamma_\varepsilon\nu(F_i)}{4CQ_d}.
\end{equation}
Note that from \eqref{eq:Enz}, each $ E_{n_z} \subset L_{n_z} $, so the union
\begin{equation}\label{eq:A}
	A:=\bigcup_{F_i\subset \cf_B}\bigcup_{L_{n_z}\in\cb(F_i)}E_{n_z}\subset \bigcup_{k= \ell}^\infty E_k\cap B
\end{equation}
is a subset of the relevant tail of the $ \limsup $ set intersected with the ball $ B $.

In the next subsection, we will construct a probability measure supported on the set $ A $, and show that $\hc^t(A) \gg \mu(B)$. This will immediately yield the desired lower bound in \eqref{eq:goal}, completing the proof of the large intersection property. Before moving to this task, we summarize several geometric and measure-theoretic properties of $ A $ established so far. These will be instrumental in the measure construction and content estimates that follow.

\begin{lem}\label{l:geoA}
	Let \( A \) be the set defined in \eqref{eq:A}. Then the following properties hold:
	\begin{enumerate}[\upshape(1)]
		\item We have the lower bound
		\begin{equation}\label{eq:m>}
			\sum_{F_i \in \cf_B} \mu(F_i) \gg \mu(B).
		\end{equation}
		Furthermore, for any two distinct sets \( F_i, F_j \in \cf_B \),
		\[
		\dist(F_i, F_j) \ge \max\{|F_i|, |F_j|\}.
		\]

		\item For each \( F_i \in \cf_B \), we have
		\begin{equation}\label{eq:mu>reformulation}
			\nu\bigg(\bigcup_{L_{n_z}\in\cb(F_i)}L_{n_z}\bigg)\gg\nu(F_i),
		\end{equation}
		where $L_{n_z}=B(z,5r_{n_z})$ is a ball with center $z\in E_{\nu,F_i}^{N,\varepsilon}$.
		Moreover, for any two distinct balls \( L_{n_z}, L_{n_w} \in \cb(F_i) \), we have
		\begin{equation}\label{eq:distE}
			L_{n_z} \cap L_{n_w} = \emptyset, \quad \text{and} \quad \dist(E_{n_z}, E_{n_w}) \ge \max\{r_{n_z}, r_{n_w}\}.
		\end{equation}
	\end{enumerate}
\end{lem}
\begin{proof}
	(1) It follows from immediately from Lemma \ref{l:kgbf}.

	(2) Equation \eqref{eq:mu>reformulation} is just a reformulation of \eqref{eq:mu>}. By the construction of the collection $\cb(F_i)$, the balls it contains are pairwise disjoint. Thus, it remains to verify the separation property:
	\[\dist(E_{n_z},E_{n_w})\ge\max\{r_{n_z},r_{n_w}\}.\]
	This follows from two observations (see \eqref{eq:condnz} and \eqref{eq:Enz}): first, the center $z$ lies in $B(x_{n_z}, r_{n_z})$; second, the set $E_{n_z}$ is contained in $B(x_{n_z}, r_{n_z})$, which in turn is contained in $L_{n_z} = B(z, 5r_{n_z})$. These nested inclusions guarantee that the sets $E_{n_z}$ are mutually disjoint and separated by at least $\max\{r_{n_z}, r_{n_w}\}$, since they are contained in disjoint balls $L_{n_z}$.
\end{proof}

\begin{rem}\label{r:not Ahlfors}
	Let us explain why it is necessary to go to such lengths to construct $A$ defined in \eqref{eq:A}. First, suppose, for the moment, that $\nu=\mu$ is the $\delta$-Ahlfors regular measure. Then, any ball $B(x,r)\subset F_i$ `cannot contain too much' sets $L_{n_z}=B(z,5r_{n_z})$, as the $\nu$-measure of $B(x,r)$ and $L_{n_z}$ is proportional to the $\delta$-power of their respective radii. This essentially means that the sets $L_{n_z}$, and hence $E_{n_z}$, are well-separated. Therefore, one could expect that the Hausdorff content of $A$, and hence the Hausdorff dimension of $\limsup E_n$, is as large as anticipated.

	 However, the situation becomes subtle when $\nu$ is singular with respect to $\mu$, and we no longer have an effective and uniform estimate for $\nu(B(x, r))$. In such cases, it could happen that $\nu(B(x, r)) = r^\alpha$, with $\alpha$ much smaller than $\hdimu \nu$, and consequently, $B(x, r)$ would `contain too many' sets $L_{n_z}$, which would, in turn, lead to $B(x, r)$ containing too many sets $E_{n_z}$. In other words, the mass of $A$ may concentrate in a small part of $A$. As a result, the Hausdorff content of $A$ would be smaller than we expected, and so would the Hausdorff dimension of $\limsup E_n$.

	 The notion of quasi-self-conformality of $\nu$ is designed to avoid this situation. As can be seen in the construction of the set $E_{\nu, F_i}^{N, \varepsilon}$, the `conformal map' $f_i$ preserves some local structure of $\nu$, thus allowing us to focus on the set of points $x \in F_i$ that satisfy the desired estimate:
	 \[\frac{\nu(B(x,\rho)\cap F_i)}{\nu(F_i)}\ll\bigg(\frac{\rho}{|F_i|}\bigg)^{\hdimu\nu-\varepsilon}.\]
\end{rem}
\subsection{Hausdorff content bound of $\bigcup_{k= \ell}^\infty E_k\cap B$}
Recall condition (3) in Theorem \ref{t:MTP},
\begin{equation}\label{eq:hcEn>}
	\hc^s(E_n)\gg r_n^{\hdimu\nu},
\end{equation}
where the implied constant is independent of $n$.
For any $n\ge 1$, by Frostman's lemma,  there exists a probability measure $\lambda_n$ supported on $E_n$ such that
\begin{equation}\label{eq:lambda}
	\lambda_n(B(x,r))\stackrel{\text{Lemma 2.2}}{\ll}\frac{r^s}{\hc^s(E_n)}\stackrel{\eqref{eq:hcEn>}}{\ll} \frac{r^s}{r_n^{\hdimu\nu}},
\end{equation}
where, again, the implied constant is independent of $n$.

Since $E_n\subset B(x_n,r_n)$, this immediately implies
\begin{equation}\label{eq:s<delta}
	s\le \hdimu\nu\le \delta.
\end{equation}
Indeed, suppose for contradiction that $s>\hdimu\nu$. Then,
\[|E_n|^s\ge \hc^s(E_n)\stackrel{\eqref{eq:hcEn>}}{\gg} r_n^{\hdimu\nu}\gg |E_n|^{\hdimu\nu}.\]
This is impossible, since $|E_n|\to 0$ as $n\to\infty$ and the implied constants are independent of $n$. Hence \eqref{eq:s<delta} must hold.

Let $A$ be defined as in \eqref{eq:A}. Define a probability measure $\eta$ supported on $A\subset \bigcup_{k= \ell}^\infty E_k\cap B$ by
\[\eta=\sum_{F_i\in \cf_B}\sum_{L_{n_z}\in\cb(F_i)}\frac{\mu(F_i)}{\sum_{F_i\subset \cf_B}\mu(F_i)}\cdot\frac{\nu(L_{n_z})}{\sum_{L_{n_z}\in\cb(F_i)}\nu(L_{n_z})}\cdot \lambda_{n_z}.\]
\begin{rem}
	The construction of the measure $\eta$ is fairly standard:
	\begin{enumerate}[\upshape(1)]
		\item the sum of the $\mu$-measures of the sets $F_i \in \mathcal{F}_B$ is comparable to that of $B$ (see \eqref{eq:m>}). We therefore assign to each $F_i$ a weight equal to the proportion of its $\mu$-measure in the total $\mu$-measure, that is,
		\[\frac{\mu(F_i)}{\sum_{F_i\subset \cf_B}\mu(F_i)};\]
		\item for each $F_i\in\cf_B$, the sum of the $\nu$-measure of the sets $L_{n_z}\in F_i$ is comparable to that of $F_i$ (see \eqref{eq:mu>reformulation}). We therefore assign to each $L_{n_z}\in F_i$ a weight
		\[\frac{\mu(F_i)}{\sum_{F_i\subset \cf_B}\mu(F_i)}\cdot\frac{\nu(L_{n_z})}{\sum_{L_{n_z}\in\cb(F_i)}\nu(L_{n_z})};\]
		\item since $L_{n_z}$ has only one descendant $E_{n_z}$, we assign the entire weight of $L_{n_z}$ to $E_{n_z}$. This together with a natrual measure $\lambda_{n_z}$ supported on $E_{n_z}$ gives the definition of $\eta$.
	\end{enumerate}
\end{rem}

Next, we estimate the $\eta$-measure of arbitrary balls, which will allow us to apply the mass distribution principle and conclude the desired lower bound on the Hausdorff content (see \eqref{eq:content bound}). Suppose that $r>0$ and
\begin{equation}\label{eq:x}
	x\in E_{n_w}\quad\text{for some\quad $E_{n_w}\subset L_{n_w}=B(w,5r_{n_w})\in\cb(F_i)$}.
\end{equation}
The separation properties of the collections $\cf_B$ and $\cb(F_i)$ (established in Lemma \ref{l:geoA}) suggest us to consider four different cases.

\noindent\textbf{Case 1:} $r > |B|$. Since $\eta$ is a probability measure supported on a subset of $B$, the measure of any ball with radius larger than $|B|$ is trivially bounded by 1. Using the $\delta$-Ahlfors regularity of $\mu$ and $s\le \delta$ (see \eqref{eq:s<delta}), we have
\begin{equation}\label{eq:case1}
	\eta(B(x,r))\le 1<\frac{r^\delta}{|B|^\delta}\ll \frac{r^s}{\mu(B)}.
\end{equation}

\noindent\textbf{Case 2:} $|F_i| \leq r < |B|$. By the separation property of the collection $\cf_B$, different sets $F_j$ are well spaced apart. Specifically, for any $F_j\in\cf_B$ distinct with $F_i$, by Lemma \ref{l:geoA} (1),
\begin{equation}\label{eq:dist}
	\dist(F_i,F_j)\ge \max\{|F_i|,|F_j|\}.
\end{equation}
If a distinct $F_j$ intersects $B(x,r)$, then its diameter must be at most $r$, implying that $F_j$ lies within a slightly larger ball $B(x, 2r)$. It follows that
\begin{align}
	\eta(B(x,r))&\le \sum_{\substack{F_i\in \cf_B\\ F_i\subset B(x,2r)}}\frac{\mu(F_i)}{\sum_{F_i\subset \cf_B}\mu(F_i)}\le \frac{\mu(B(x,2r))}{\sum_{F_i\subset \cf_B}\mu(F_i)}\stackrel{\eqref{eq:m>}}{\ll}\frac{r^\delta}{\mu(B)}\notag\\
	& \stackrel{\eqref{eq:s<delta}}{\le}\frac{r^s}{\mu(B)}.\label{eq:case2}
\end{align}
\noindent\textbf{Case 3:} $r_{n_w} \leq r < |F_i|$. Here, $B(x,r)$ intersects only one $F_i$ because of the minimal distance (see \eqref{eq:dist}) between distinct sets $F_i$. We break it down into two subcases:

\noindent\textbf{Subcase 3a:} $r \geq |F_i|/(16C N)$.
By the definition of $\eta$,
\begin{equation}\label{eq:case3a}
	\eta(B(x,r))\le \frac{\mu(F_i)}{\sum_{F_i\subset \cf_B}\mu(F_i)}\stackrel{\eqref{eq:m>}}{\ll}\frac{|F_i|^\delta}{\mu(B)}\le \frac{(16C N r)^\delta}{\mu(B)}\stackrel{\eqref{eq:s<delta}}{\ll} \frac{r^s}{\mu(B)},
\end{equation}
where we use the fact that $N=N(\varepsilon)$ is independent of $B$ (see \eqref{eq:gammaeps}).

\noindent\textbf{Subcase 3b}: $r_{n_w}\le r<|F_i|/(16C N)$. For any ball $L_{n_z}\in\cb(F_i)$ distinct with $L_{n_w}$, if $B(x,r)\cap E_{n_z}=\emptyset$, then
\[\lambda_{n_z}(B(x,r))=0,\]
since $\lambda_{n_z}$ is supported on the set $E_{n_z}$.
Consequently, we have
\[\begin{split}
	\eta(B(x,r))&\le \sum_{\substack{L_{n_z}\in\cb(F_i)\\B(x,r)\cap E_{n_z}\ne\emptyset}}\frac{\mu(F_i)}{\sum_{F_i\subset \cf_B}\mu(F_i)}\cdot\frac{\nu(L_{n_z})}{\sum_{L_{n_z}\in\cb(F_i)}\nu(L_{n_z})}\\
	&\stackrel{\text{\eqref{eq:m>} and  \eqref{eq:mu>reformulation}}}{\ll}\sum_{\substack{L_{n_z}\in\cb(F_i)\\B(x,r)\cap E_{n_z}\ne\emptyset}}\frac{\mu(F_i)}{\mu(B)}\cdot\frac{\nu(L_{n_z})}{\nu(F_i)}.
\end{split}\]
 Note that by \eqref{eq:distE} the sets $E_{n_z}$ and $E_{n_w}$ are well-separated:
\begin{equation}\label{eq:distEw}
	\dist(E_{n_z},E_{n_w})\ge\max\{r_{n_z},r_{n_w}\}.
\end{equation}
Therefore, if $B(x,r)\cap E_{n_z}\ne\emptyset$, then it must be that
\[\begin{split} r>r_{n_z}\ge|L_{n_z}|/10\quad\Longrightarrow\quad E_{n_z}\subset L_{n_z}\subset B(x,11r)\subset B(w,16r),
\end{split}\]
where the last inclusion uses the fact that $x\in E_{n_w}\subset B(w,5r_{n_w})$ (see \eqref{eq:x}). It then follows that
\[\begin{split}
	\eta(B(x,r))&\ll\sum_{\substack{L_{n_z}\in\cb(F_i)\\B(x,r)\cap E_{n_z}\ne\emptyset}}\frac{\mu(F_i)}{\mu(B)}\cdot\frac{\nu(L_{n_z})}{\nu(F_i)}.\ll\sum_{\substack{L_{n_z}\in\cb(F_i)\\ L_{n_z}\subset B(w,16r)}}\frac{\mu(F_i)}{\mu(B)}\cdot\frac{\nu(L_{n_z})}{\nu(F_i)}\\
	&\le\frac{\mu(F_i)}{\mu(B)}\cdot\frac{\nu(B(w,16r)\cap F_i)}{\nu(F_i)}.
\end{split}\]
Note that $w\in E_{\nu,F_i}^{N,\varepsilon}$ by our construction (see Lemma \ref{l:geoA} (2) and \eqref{eq:x}). Since $r<|F_i|/(16C N)$ (equivalently $16 r<|F_i|/(CN)$), by the definition of $E_{\nu,F_i}^{N,\varepsilon}$,
\begin{equation}\label{eq:<dimHnu-e}
	\frac{\mu(B(w,16r)\cap F_i)}{\nu(F_i)}\ll \bigg(\frac{16r}{|F_i|}\bigg)^{\hdimu\nu-\varepsilon}\stackrel{\eqref{eq:s<delta}}{\ll} \frac{r^{s-\varepsilon}}{|F_i|^\delta}.
\end{equation}
Putting these estimates together, we conclude that
\begin{equation}\label{eq:case3b}
	\eta(B(x,r))\ll\frac{\mu(F_i)}{\mu(B)}\cdot\frac{\nu(B(w,16r)\cap F_i)}{\nu(F_i)}\ll\frac{\mu(F_i)}{\mu(B)}\cdot \frac{r^{s-\varepsilon}}{|F_i|^\delta}\asymp\frac{r^{s-\varepsilon}}{\mu(B)}.
\end{equation}

\noindent\textbf{Case 4}: $0< r<r_{n_w}$. In this scale, due to the separation of the sets $E_{n_z}$ (see \eqref{eq:distE}), the ball $B(x,r)$ can intersect only the single set $E_{n_w}$ containing $x$. Note that the discussion of Case 3 implies that for $r=r_{n_w}$ (see \eqref{eq:<dimHnu-e}),
\[\frac{\mu(F_i)}{\sum_{F_i\subset \cf_B}\mu(F_i)}\cdot\frac{\nu(L_{n_z})}{\sum_{L_{n_z}\in\cb(F_i)}\nu(L_{n_z})}\ll \frac{r_{n_w}^{\hdimu\nu-\varepsilon}}{\mu(B)}.\]
By the Frostman-type property for $\lambda_{n_w}$ (see \eqref{eq:lambda}),
\begin{align}
		\eta(B(x,r))&\le \frac{\mu(F_i)}{\sum_{F_i\subset \cf_B}\mu(F_i)}\cdot\frac{\nu(L_{n_z})}{\sum_{L_{n_z}\in\cb(F_i)}\nu(L_{n_z})}\cdot\lambda_{n_w}(B(x,r))\notag\\
		&\ll \frac{r_{n_w}^{\hdimu\nu-\varepsilon}}{\mu(B)}\cdot \lambda_{n_w}(B(x,r))\stackrel{\eqref{eq:lambda}}{\ll}\frac{r_{n_w}^{\hdimu\nu-\varepsilon}}{\mu(B)}\cdot \frac{r^s}{r_{n_w}^{\hdimu\nu}}\notag\\
	&=\frac{r_{n_w}^{-\varepsilon}r^s}{\mu(B)}\le \frac{r^{s-\varepsilon}}{\mu(B)},\label{eq:case4}
\end{align}
where the last inequality follows from $0<r<r_{n_w}$.

By Cases 1--4, we have for any ball $B(x,r)$,
\[\eta(B(x,r))\ll\frac{r^{s-\varepsilon}}{\mu(B)}=\frac{r^t}{\mu(B)},\]
where the equality follows from $\varepsilon=s-t$. Since $\eta$ is supported on $A\subset \bigcup_{k= \ell}^\infty E_k\cap B$, by the mass distribution principle,
\[\hc^t\bigg(\bigcup_{k= \ell}^\infty E_k\cap B\bigg)\gg \mu(B).\]
With the discussion at the beginning of Section \ref{s:MTP}, the proof of Theorem \ref{t:MTP} is now complete.

\section{Application to $\beta$-transformations}\label{s:beta}
\subsection{Definition and some basic properties}
For $\beta>1$, the $\beta$-transformation $T_\beta:[0,1)\to[0,1)$ is defined by
\[T_\beta x=\beta x\ (\textrm{mod}\ 1).\]
For any $ n\ge 1 $ and $ x\in[0,1) $, define $ \epsilon_n(x,\beta)=\lfloor \beta T_\beta^{n-1}x\rfloor $, where $\lfloor x\rfloor$ denotes the integer part of $x$. Then, we can write
\[x=\frac{\epsilon_1(x,\beta)}{\beta}+\frac{\epsilon_2(x,\beta)}{\beta^2}+\cdots+\frac{\epsilon_n(x,\beta)}{\beta^n}+\cdots,\]
and we call the sequence
\[\epsilon(x,\beta):=(\epsilon_1(x,\beta),\epsilon_2(x,\beta),\dots)\]
the $\beta$-expansion of $ x $. By the definition of $ T_\beta $, it is clear that, for $ n\ge 1 $, $ \epsilon_n(x,\beta) $ belongs to the alphabet $ \{0,1,\dots,\lceil\beta-1\rceil\} $, where $ \lceil x\rceil $ denotes the smallest integer greater than or equal to $ x $. When $ \beta $ is not an integer, then not all sequences of $ \{0,1,\dots,\lceil\beta-1\rceil\}^\N $ are the $ \beta $-expansion of some $ x\in[0,1) $. This leads to the notion of $\beta$-admissible sequence.

\begin{defn}
	A finite or an infinite sequence $ (\epsilon_1,\epsilon_2,\dots)\in\{0,1,\dots,\lceil\beta-1\rceil\}^\N $ is said to be $\beta$-admissible if there exists an $ x\in[0,1) $ such that the $\beta$-expansion of $ x $ begins with $ (\epsilon_1,\epsilon_2,\dots) $.
\end{defn}

Denote by $ \Sigma_\beta^n $ the collection of all admissible sequences of length $ n $. The following result of R\'enyi~\cite{RenyibetaAMH} implies that the dynamical system $ ([0,1],T_\beta) $ admits $ \log\beta $ as its topological entropy.
\begin{lem}[{\cite[(4.9) and (4.10)]{RenyibetaAMH}}]\label{l:renyi}
	Let $ \beta>1 $. For any $ n\ge 1 $,
	\[\beta^n\le \# \Sigma_\beta^n\le \frac{\beta^{n+1}}{\beta-1},\]
	where $ \# $ denotes the cardinality of a finite set.
\end{lem}

\begin{defn}
	For any $ \be_n:=(\epsilon_1,\dots,\epsilon_n)\in\Sigma_\beta^n $, we call
	\[I_{n,\beta}(\be_n):=\{x\in[0,1):\epsilon_k(x,\beta)=\epsilon_k,1\le k\le n\}\]
	an $ n $th level cylinder.
\end{defn}
Each cylinder $ I_{n,\beta}(\be_n) $ can be viewed as a subinterval of $[0,1)$ consisting of all points whose first $n$ digits in their $\beta$-expansion coincide with the word $\be_n$. These cylinders form a natural partition of the interval $[0,1)$ at level $n$, and they shrink as $n$ increases.
Clearly, for any $\be_n \in \Sigma_\beta^n$, the map $T_\beta^n$ is linear with slope $\beta^n$ when restricted to the cylinder $I_{n,\beta}(\be_n)$, and it sends $I_{n,\beta}(\be_n)$ into $[0,1)$.
If $\beta$ is not an integer, then the dynamical system $(T_\beta, [0,1))$ is not a full shift, and thus $T_\beta^n|_{I_{n,\beta}(\be_n)}$ may fail to be onto $[0,1)$.
In other words, the length of $I_{n,\beta}(\be_n)$ may be strictly less than $\beta^{-n}$, which complicates the analysis of the dynamical properties of $T_\beta$.
In many cases, including the one considered here, it is more convenient to restrict attention to cylinders of maximal length, which motivates the definition of {\em full cylinder}.
\begin{defn}
	A cylinder $ I_{n,\beta}(\be_n) $ or a sequence $ \be_n\in\Sigma_\beta^n $ is called full if it has maximal length, that is, if
	\[|I_{n,\beta}(\be_n)|=\beta^{-n}.\]
\end{defn}
The key property of full sequences is that the concatenation of any two full sequences is again full.
	\begin{prop}[{\cite[Lemma 3.2]{FanWangbetaNon}}]\label{p:concatenation}
	An $ n $th level cylinder $ I_{n,\beta}(\be_n) $ is full if and only if, for any $\beta$-admissible sequence $ \be'_m\in\Sigma_\beta^m $ with $ m\ge 1 $, the concatenation $ \be_n\be_m' $ is still $ \beta $-admissible. Moreover,
	\[|I_{n+m,\beta}(\be_n\be_m')|=|I_{n,\beta}(\be_n)|\cdot|I_{m,\beta}(\be_m')|.\]
\end{prop}
A corollary in \cite{FanWangbetaNon} establishes a certain relationship between full and non-full cylinders. While the statement given below is not stated exactly as in their original work, it can nonetheless be rigorously justified using the same method.
\begin{lem}[{\cite[Corollary 3.3 and (3.4)]{FanWangbetaNon}}]\label{l:0^kfull}
	Let $n\in \N$ and $\be_n\in\Sigma_\beta^n$. Let $m\ge n$ be the unique integer satisfying
	\[\beta^{-m-1}<|I_{n,\beta}(\be_n)|\le\beta^{-m}.\]
	Then, $I_{n,\beta}(\be_n) = I_{m,\beta}(\be_n, 0^{m-n})$ and $I_{m+1}(\be_n, 0^{m-n+1})\subset I_{n,\beta}(\be_n) $ is full, where $0^k$ denotes the word consisting of $k$ consecutive zeros.
\end{lem}

In light of the definition of quasi-self-conformality, the collection $\mathcal{F}$ of sets required therein can naturally be taken to be the family of full cylinders. Moreover, in order to apply Theorem \ref{t:MTP}, it is necessary that the $\limsup$ set defined by full cylinders has full Lebesgue measure. Fortunately, this is indeed the case.
\begin{lem}[{\cite[Lemma 1 (1)]{WaltersbetaGibbs}}]\label{l:union=[0,1]}
	For any $N\ge 1$, we have
	\[\bigcup_{n=N}^\infty\bigcup_{\be_n\in\Lambda_{\beta}^n}I_{n,\beta}(\be_n)=[0,1),\]
	where $\Lambda_\beta^n$ denotes the set of full sequences of length $n$.
	In particular, the $\limsup$ set defined by all full cylinders has full Lebesgue measure.
\end{lem}
The mass distribution principle stated in Proposition \ref{p:MDP} requires estimating the measure of arbitrary balls in relation to their radii. However, after a detailed study of the distribution of full cylinders, Bugeaud and Wang \cite[Proposition 1.3]{BugeaudWangSTPJFG} showed that it suffices to consider balls that are themselves cylinders.
\begin{prop}[Modified mass distribution principle {\cite[Proposition 1.3]{BugeaudWangSTPJFG}}]\label{p:modifiedMDP}
	Let $E$ be a Borel measurable set in $[0,1]$ and $\lambda$ be a Borel measure with $\lambda(E)>0$. Assume that there exist a positive constant $c>0$ and an integer $n_0$ such that,  for any $n\ge n_0$ the measure of any cylinder $I_{n,\beta}(\be_n)$ of order $n$ satisfies $\lambda(I_{n,\beta}(\be_n))\le c|I_{n,\beta}(\be_n)|^s$. Then,  $\hdim E\ge s$.
\end{prop}

\subsection{Pressure function}\label{ss:presure function beta}
Let $\phi:[0,1]\to\R$ be a continuous function. The pressure function for $\beta$-dynamical system associated to $\phi$ is defined by the limit
\begin{equation}\label{eq:pressurelimitbeta}
	P(\phi,T_\beta):=\lim_{n\to\infty}\frac{1}{n}\log\sum_{\be_n\in\Sigma_\beta^n}e^{S_n\phi(y)},
\end{equation}
where for each admissible word $\be_n\in\Sigma_\beta^n$, the point $y$ is any element in the corresponding cylinder $I_{n,\beta}(\be_n)$, and $S_n\phi(y)$ denotes the ergodic sum $\sum_{k=0}^{n-1}\phi(T_\beta^ky)$. The existence of the limit in \eqref{eq:pressurelimitbeta} follows from the subadditivity:
\[\log\sum_{\be_{n+m}\in\Sigma_\beta^{n+m}}e^{S_{n+m}\phi(y)}\le \log\sum_{\be_{n}\in\Sigma_\beta^n}e^{S_n\phi(y)}+\log\sum_{\be_{m}\in\Sigma_\beta^{m}}e^{S_{m}\phi(T^ny)},\]
and the limit does not depend on the choice of $y$ by the continuity of $\phi$.

It follows directly from the definition that the pressure function is continuous with respect to $\phi$. In the absence of a suitable reference, we provide a proof for the sake of completeness.
\begin{prop}\label{p:continuous beta}
	Let $\phi$ and $\varphi$ be two continuous functions defined on $[0,1]$. Then,
	\[|P(\phi,T_\beta)-P(\varphi,T_\beta)|\le \|\phi-\varphi\|_\infty,\]
	where $\|\cdot\|_{\infty}$ denotes the supremum norm of a function in $C^0([0,1])$.
	Consequently, if $\phi$ is positive, then there exists $0< s=s(\phi)< 1$ such that
	\[P(-s(\phi+\log\beta),T_\beta)=0.\]
\end{prop}
\begin{proof}
	By the definition of $P(\phi,T_\beta)$ (see \eqref{eq:pressurelimitbeta}), for each $n\in\N$, we have
	\[\begin{split}
		\frac{1}{n}\log\sum_{\be_n\in\Sigma_\beta^n}e^{S_n\phi(y)}&=\frac{1}{n}\log\sum_{\be_n\in\Sigma_\beta^n}e^{S_n\varphi(y)+S_n\phi(y)-S_n\varphi(y)}\\
		&\le \frac{1}{n}\log\sum_{\be_n\in\Sigma_\beta^n}e^{S_n\varphi(y)+n\|\phi-\varphi\|_\infty}\\
		&=\bigg(\frac{1}{n}\log\sum_{\be_n\in\Sigma_\beta^n}e^{S_n\varphi(y)}\bigg) +\|\phi-\varphi\|_\infty.
	\end{split}\]
	Letting $n\to\infty$, it follows that
	\[P(\phi,T_\beta)-P(\varphi,T_\beta)\le \|\phi-\varphi\|_\infty.\]
	Interchanging the roles of $\phi$ and $\varphi$ immediately yields the first point of the proposition.

	By the above conclusion, the function $s\mapsto P(-s(\phi+\log\beta),T_\beta)$ is continuous in $s$. By the intermediate value theorem, to conclude the second point of the proposition, it suffices to prove that $P(-s(\phi+\log\beta),T_\beta)$ takes values of opposite signs at $s=0$ and $s=1$. For $s=0$, we have
	\[P(0,T_\beta)=\lim_{n\to\infty}\frac{1}{n}\log\sum_{\be_n\in\Sigma_\beta^n}e^0=\lim_{n\to\infty}\frac{1}{n}\log\#\Sigma_\beta^n\stackrel{\text{Lemma \ref{l:renyi}}}{=}\log\beta>0.\]
	For $s=1$, since $\phi$ is positive, we have
	\[\begin{split}
		P(-(\phi+\log\beta),T_\beta)&=\lim_{n\to\infty}\frac{1}{n}\log\sum_{\be_n\in\Sigma_\beta^n}e^{-S_n(\phi+\log\beta)(y)}\\
		&= \lim_{n\to\infty}\frac{1}{n}\log\sum_{\be_n\in\Sigma_\beta^n}e^{-S_n\phi(y)-n\log\beta}\le \lim_{n\to\infty}\frac{1}{n}\log\sum_{\be_n\in\Sigma_\beta^n}e^{-n\log\beta}\\
		&=\lim_{n\to\infty}\frac{1}{n}\log\big(\beta^{-n}\cdot\#\Sigma_\beta^n\big)\stackrel{\text{Lemma \ref{l:renyi}}}{=}0,
	\end{split}\]
	which is what we want.
\end{proof}
Guided by Theorem \ref{t:MTP}, it is necessary to choose a reference measure $\nu$ --- generally singular with respect to the Lebesgue measure --- to measure the size of the $\limsup$ sets. Such a measure is usually chosen as the Gibbs measure, whose existence is ensured by the following result.

\begin{thm}[{\cite[Theorems 13 and 16]{WaltersbetaGibbs}}]\label{t:Gibbsbeta}
	Let $\phi:[0,1] \to \mathbb{R}$ be a Lipschitz continuous function. Then there exists a unique equilibrium state $\nu_\phi$ associated with $\phi$ such that the following properties hold:
	\begin{enumerate}[(1)]
		\item The pressure satisfies the variational principle:
		\[
		P(\phi, T_\beta) = h_{\nu_\phi} + \int \phi \, d\nu_\phi,
		\]
		where $h_{\nu_\phi}$ is the measure-theoretic entropy of $\nu_\phi$ with respect to $T_\beta$.

		\item For any cylinder $I_{n,\beta}(\be_n)$ of level $n$, the measure $\nu_\phi$ satisfies the upper Gibbs property:
		\[
		\nu_\phi(I_{n,\beta}(\be_n)) \ll e^{S_n \phi(x) - n P(\phi, T_\beta)},
		\]
		where $x \in I_{n,\beta}(\be_n)$ is arbitrary. Moreover, if the cylinder $I_{n,\beta}(\be_n)$ is full, then the Gibbs property holds in the sense:
		\[
		\nu_\phi(I_{n,\beta}(\be_n)) \asymp e^{S_n \phi(x) - n P(\phi, T_\beta)}.
		\]
		Here, the implied constants do not depend on the particular choice of $x \in I_{n,\beta}(\bm\epsilon_n)$.
	\end{enumerate}
\end{thm}

It is important to note that in condition (3) of Theorem \ref{t:MTP}, one needs to compare the Hausdorff content $\hc^s(E_n)$ with $r_n^{\hdimu \nu_\phi}$. This comparison relies on understanding the Hausdorff dimension of the Gibbs measure $\nu_\phi$.  The author believes that the following result has been established elsewhere; however, since no suitable reference could be found, we include the proof here.
\begin{lem}\label{l:dimGibbsbeta}
		Let $\phi:[0,1]\to\R$ be a Lipschitz continuous function and $\nu_\phi$ be the associated equilibrium state. Then,
		\[\hdim \nu_\phi= \frac{h_{\nu_\phi}}{\log\beta}.\]
\end{lem}
\begin{proof}
	For any $n\ge 1$, denote by $I_{n,\beta}(x)$ the cylinder of level $n$ that contains $x$. Obviously,
	\[I_{n,\beta}(x)\subset B(x,|I_{n,\beta}(x)|),\]
	and so
	\[\frac{\log\nu_\phi(B(x,|I_{n,\beta}(x)|))}{\log |I_{n,\beta}(x)|}\le \frac{\log\nu_\phi(I_{n,\beta}(x))}{\log|I_{n,\beta}(x)|}.\]
	Note that $|I_{n,\beta}(x)|$ goes to $0$ as $n\to\infty$. By the definition of local dimension,
	\[\begin{split}
		\underline{D}(\nu_\phi,x)&=\liminf_{r\to 0}\frac{\log\nu_\phi(B(x,r))}{\log r}\le\liminf_{n\to\infty}\frac{\log\nu_\phi(B(x,|I_{n,\beta}(x)|))}{\log |I_{n,\beta}(x)|}\\
		&\le\liminf_{n\to\infty}\frac{\log\nu_\phi(I_{n,\beta}(x))}{\log|I_{n,\beta}(x)|}\le \liminf_{\substack{n\to\infty\\ I_{n,\beta}(x)\text{ is full}}}\frac{\log\nu_\phi(I_{n,\beta}(x))}{\log|I_{n,\beta}(x)|},
	\end{split}\]
	where, in the last inequality, we use the fact that for any $x\in[0,1)$, there exists infinitely many $n$ such that $I_{n,\beta}(x)$ is full (see Lemma \ref{l:union=[0,1]}).

	By Birkhoff's ergodic theorem, for $\nu_\phi$-almost all $x$,
	\begin{equation}\label{eq:ergodicbeta}
		\lim_{n\to\infty}\frac{1}{n}S_n\phi(x)=\int\phi\dif\nu_\phi.
	\end{equation}
	Let $x\in[0,1]$ be such that \eqref{eq:ergodicbeta} holds. For any full cylinder $I_{n,\beta}(x)$, by the Gibbs property (see Theorem \ref{t:Gibbsbeta} (2)),
	\[\nu_\phi(I_{n,\beta}(x))\asymp e^{S_n\phi(x)-nP(\phi,T_\beta)}\stackrel{\eqref{eq:ergodicbeta}}{=}e^{n(\int\phi\dif\nu_\phi-P(\phi,T_\beta)+o(1))},\]
	where $o(1)\to 0$ as $n\to\infty$. Therefore,
	\[\begin{split}
		\underline{D}(\nu_\phi,x)&\le \liminf_{\substack{n\to\infty\\ I_{n,\beta}(x)\text{ is full}}}\frac{n\big(\int\phi\dif\nu_\phi-P(\phi,T_\beta)+o(1)\big)}{\log|I_{n,\beta}(x)|}=\frac{P(\phi,T_\beta)-\int\phi\dif\nu_\phi}{\log\beta}\\
		&=\frac{h_{\nu_\phi}}{\log\beta},
	\end{split}\]
	where the equality follows from the variational principle (see Theorem \ref{t:Gibbsbeta} (1)).
	Since this holds for $\nu_\phi$-almost all $x$, we have
	\begin{equation}\label{eq:uppermudimbeta}
		\hdimu\nu_\phi\le\frac{h_{\nu_\phi}}{\log\beta}.
	\end{equation}

	Next, we prove the reverse inequality. For any set $E$ with positive $\nu_\phi$-measure, we claim that
	\begin{equation}\label{eq:claimEdim}
		\hdim E\ge \frac{h_{\nu_\phi}}{\log\beta}.
	\end{equation}
	It then follows from the definition of $\hdiml\nu_\phi$ that
	\[\hdiml\nu_\phi\ge \frac{h_{\nu_\phi}}{\log\beta},\]
	which together with \eqref{eq:uppermudimbeta} concludes the proof.

	Fix a Borel set $E$ with $\nu_\phi(E)>0$. Let $\varepsilon>0$. Then, in view of \eqref{eq:ergodicbeta}, there exists an integer $N=N(\varepsilon)$ such that the set $A_N$ of $x$ for which
	\[\bigg|\frac{1}{n}S_n\phi(x)-\int\phi\dif\nu_\phi\bigg|<\varepsilon,\quad \text{for all $n\ge N$}\]
	has $\nu_\phi$-measure larger than $1-\nu_\phi(E)$. Obviously, $A_N\cap E$ has positive $\nu_\phi$-measure.

	Let $\lambda=\nu_\phi|_{A_N\cap E}$. For any $n\ge N$ and any cylinder $I_{n,\beta}(\be_n)$, if
	\[\bigg|\frac{1}{n}S_n\phi(x)-\int\phi\dif\nu_\phi\bigg|\ge\varepsilon\]
	for all $x\in I_{n,\beta}(\be_n)$. Then,
	\[\lambda(I_{n,\beta}(\be_n))=\nu_\phi(A_N\cap E\cap I_{n,\beta}(\be_n))=0.\]
	Therefore, for any cylinder $I_{n,\beta}(\be_n)$,
	\begin{align}
			&\lambda(A_N\cap E\cap I_n(\be_n))>0\notag\\
		\Longrightarrow\quad&\bigg|\frac{1}{n}S_n\phi(x)-\int\phi\dif\nu_\phi\bigg|<\varepsilon\text{ for some $x\in I_{n,\beta}(\be_n)$}.\label{eq:positiveimplyergodicsum}
	\end{align}
	We stress that the reverse implication may not be true.
	Let $I_{n,\beta}(\be_n)$ be a cylinder with positive $\nu_\phi$-measure and let $x=x(\be_n)$ be such that  \eqref{eq:positiveimplyergodicsum} holds. Let $m\ge n$ be the unique integer satisfying
	\[\beta^{-m-1}<|I_{n,\beta}(\be_n)|\le\beta^{-m}.\]
	By Lemma \ref{l:0^kfull}, we have $I_{n,\beta}(\be_n) = I_{m,\beta}(\be_n, 0^{m-n})$. Since $x\in I_{n,\beta}(\be_n)=I_{m,\beta}(\be_n,0^{m-n})$, it follows from Theorem \ref{t:Gibbsbeta} (2) that
	\[\begin{split}
		\lambda(I_{n,\beta}(\be_n))&\le \nu_\phi(I_{m,\beta}(\be_n,0^{m-n}))\ll e^{S_m\phi(x)-mP(\phi,T_\beta)}\\
		&\stackrel{\eqref{eq:positiveimplyergodicsum}}{\le} e^{m\int\phi\dif\nu_\phi+m\varepsilon-mP(\phi,T_\beta)}
		=\beta^{-m(P(\phi,T_\beta)-\int\phi\dif\nu_\phi-\varepsilon)/\log\beta}\\
		&\asymp |I_{n,\beta}(\be_n)|^{(P(\phi,T_\beta)-\int\phi\dif\nu_\phi-\varepsilon)/\log\beta}.
	\end{split}\]
	By Proposition \ref{p:modifiedMDP},
	\[\hdim (A_N\cap E)\ge \frac{P(\phi,T_\beta)-\int\phi\dif\nu_\phi-\varepsilon}{\log\beta}=\frac{h_{\nu_\phi}-\varepsilon}{\log\beta},\]
	where the equality follows from Theorem \ref{t:Gibbsbeta} (1).
	By the arbitrariness of $\varepsilon$, the claim \eqref{eq:claimEdim} follows immediately.
\end{proof}
\subsection{Application to shrinking target problems}\label{ss:proofbeta}
Recall that $\{h_n\}$ is a sequence of Lipschitz functions with uniformly bounded Lipschitz constants, and that
\[W(T_\beta,f, \{h_n\})=\{x\in[0,1):|T_\beta^nx-h_n(x)|<e^{-S_nf(x)}\text{ for i.m. $n$}\},\]
where $f:[0,1]\to\R$ is a positive continuous function. In this section, we will prove that
\begin{equation}\label{eq:stpbeta}
	W(T_\beta,f,\{h_n\})\in\mathscr{G}^s([0,1]),
\end{equation}
where $s$ satisfies $P(-s(f+\log\beta),T_\beta)=0$.
The proof relies on the following simple but useful fact.
	\begin{lem}[\cite{WangSTPJMAA}]\label{l:length}
	Let $h:[0,1)\to[0,1)$ be a  Lipschitz function with Lipschitz constant $L\ge0$. Let $0<r<1$. For any $n$ with $ L<\beta^n $ and any sequence $\be_n\in\Sigma_\beta^n$, the set
	\[\{x\in I_{n,\beta}(\be_n):|T_\beta^nx-h(x)|<r\}\]
	is contained in a ball of radius $2r\beta^{-n}$. Moreover, if $\be_n$ is full, then it contains a ball of radius $r\beta^{-n}/2$.
\end{lem}
Let us turn to the proof of \eqref{eq:stpbeta}.
\begin{proof}[Proof of \eqref{eq:stpbeta}]
	Note that Lipschitz functions are dense in $C^0([0,1])$ and that
	\[W(T_\beta,g,\{h_n\})\subset W(T_\beta,f,\{h_n\}) \quad\text{whenever $g\ge f$}.\]
	Since the pressure function $P(\cdot,T_\beta)$ is continuous (see Proposition \ref{p:continuous beta}), we can choose a Lipschitz function $g$ so that the Hausdorff dimension of $W(T_\beta,g,\{h_n\})$ is as close as we want to that of $W(T_\beta,f,\{h_n\})$, provided \eqref{eq:stpbeta} holds. Therefore, in what follows, we may assume that $f$ is Lipschitz.

	Let $\nu_s$ be the Gibbs measure associated to the Lipschitz function $-s(f+\log\beta)$. Since the concatenation of any two full sequences is still full (see Proposition \ref{p:concatenation}), by the Gibbs property of $\nu_s$ (see Theorem \ref{t:Gibbsbeta}), it is not difficult to verify that $\nu_s$ is quasi-self-conformal with respect to the collection of full cylinders. Moreover, by Lemma \ref{l:union=[0,1]}, the $\limsup$ set defined by the collection of full cylinders has full Lebesgue measure. Therefore, Theorem \ref{t:MTP} can be applied to $\nu_s$.

	Birkhoff's ergodic theorem gives that for $\nu_s$-almost all $x$,
	\[\lim_{n\to\infty}\frac{1}{n}S_n f(x)=\int f\dif\nu_s.\]
	Since for any $x\in[0,1)$, there exist infinitely many $n$ such that $I_{n,\beta}(x)$ is full, we have that the set
	\[\bigcap_{N=1}^\infty\bigcup_{n=N}^\infty \bigcup_{\be_n\in\Lambda_\beta^n(\nu_s,\varepsilon)}I_{n,\beta}(\be_n)\]
	is of full $\nu_s$-measure, where recall that $\Lambda_\beta^n$ is the set of full sequences of length $n$, and
	\[\Lambda_\beta^n(\nu_s,\varepsilon):=\bigg\{\be_n\in\Lambda_\beta^n:\bigg|\frac{1}{n}S_nf(x)-\int f\dif\nu_s\bigg|<\varepsilon\text{ for all $x\in I_{n,\beta}(\be_n)$}\bigg\}.\]
	Then,
	\[\begin{split}
		W(T_\beta,f,\{h_n\})\supset\bigcap_{N=1}^\infty\bigcup_{n=N}^\infty \bigcup_{\be_n\in\Lambda_\beta^n(\nu_s,\varepsilon)}I_{n,\beta}(\be_n)\cap E_n(T_\beta,f,h_n),
	\end{split}\]
	where $E_n(T_\beta,f,h_n)=\{x\in[0,1):|T_\beta^nx-h_n(x)|<e^{-S_nf(x)}\}$. By the Lipschitz continuity of $f$, for any $x,y\in I_{n,\beta}(\be_n)$,
	\[|S_nf(x)-S_nf(y)|\ll 1\quad\Longrightarrow\quad e^{-S_nf(x)}\asymp e^{-S_nf(y)},\]
	where the implied constant is absolute.
	Therefore, since $\be_n\in \Lambda_\beta^n(\nu_s,\varepsilon)$ is full, we can apply Lemma \ref{l:length} to conclude that $I_{n,\beta}(\be_n)\cap E_n(T_\beta,f,h_n)$ contains an interval of length
	\begin{equation}\label{eq:contain an interval beta}
		\asymp  \beta^{-n}e^{-S_nf(x)}\ge\beta^{-n}e^{-n\int f\dif\nu_s-n\varepsilon}=\beta^{-n(1+(\int f\dif\nu_s+\varepsilon)/\log\beta)}.
	\end{equation}
	By the variational principle (see Theorem \ref{t:Gibbsbeta} (1)) and note that $P(-s(f+\log\beta),T_\beta)=0$, we have
	\[0=h_{\nu_s}-s\bigg(\int f\dif\nu_s+\log\beta\bigg)\quad\Longrightarrow\quad s=\frac{h_{\nu_s}}{\int f\dif\nu_s+\log\beta}.\]
Let $K=\frac{h_{\nu_s}}{(\int f\dif\nu_s+\log\beta)\log\beta}$. Then,
		\begin{align}
			&\bigg(1+\frac{\int f\dif\nu_s+\varepsilon}{\log\beta}\bigg)\bigg(\frac{h_{\nu_s}}{\int f\dif\nu_s+\log\beta}-K\varepsilon\bigg)\notag\\
			=& \frac{h_{\nu_s}}{\log\beta}+\frac{\varepsilon h_{\nu_s}}{(\int f\dif\nu_s+\log\beta)\log\beta}-\bigg(1+\frac{\int f\dif\nu_s+\varepsilon}{\log\beta}\bigg)\cdot K\varepsilon\notag\\
			\le& \frac{h_{\nu_s}}{\log\beta}+\frac{\varepsilon h_{\nu_s}}{(\int f\dif\nu_s+\log\beta)\log\beta}-K\varepsilon=\frac{h_{v_s}}{\log\beta}.\label{eq:s-ke}
		\end{align}
	This together with Lemma \ref{l:dimGibbsbeta} yields
	\[\begin{split}
		\hc^{s-K\varepsilon}\big(I_{n,\beta}(\be_n)\cap E_n(T_\beta,f,h_n)\big)&\stackrel{\eqref{eq:contain an interval beta}}{\asymp} \beta^{-n(1+(\int f\dif\nu_s+\varepsilon)/\log\beta)(s-K\varepsilon)}\\
		&\stackrel{\eqref{eq:s-ke}}{\ge} \beta^{-nh_{\nu_s}/\log\beta}=|I_{n,\beta}(\be_n)|^{-\hdim\nu_s}.
	\end{split}\]
	Therefore, by Theorem \ref{t:MTP}, we have
	\[W(T_\beta,f, \{h_n\})\in\mathscr{G}^{s-K\varepsilon}([0,1]).\]
	Since $\varepsilon$ is arbitrary and $K$ does not depend on $\varepsilon$,
		\[W(T_\beta,f, \{h_n\})\in\mathscr{G}^{s}([0,1]).\qedhere\]
\end{proof}
\begin{rem}\label{r:applicable to Markov}
	In order to highlight further applications, such as to expanding Markov maps with finite partitions on $[0,1]$, of Theorem \ref{t:MTP}, we  summarize the properties used in the proof of \eqref{eq:stpbeta} as follows:
	\begin{enumerate}
		\item There exists $s\ge 0$ such that $P(-s(f+\log\beta),T_\beta)=0$.
		\item For the pressure function $-s(f+\log\beta)$ given in item (1), there exists a Gibbs measure $\nu_s$ verifies the variational principle
		\begin{equation}\label{eq:variational principle}
			0=P(-s(f+\log\beta), T_\beta) = h_{\nu_s} + \int -s(f+\log\beta)\, d\nu_s.
		\end{equation}
		The Gibbs property of $\nu_s$ ensures that this measure is quasi-self-conformal with respect to the collection of all cylinders.
		\item The dimension of $\nu_s$ satisfies
		\[\hdim\nu_s=\frac{ h_{\nu_s} }{\log\beta}\stackrel{\eqref{eq:variational principle}}{=}\frac{ \int s(f+\log\beta)\, d\nu_s }{\int\log|T_\beta'|\, d\nu_s},\]
		where $T_\beta'$ denotes the derivative of $T_\beta$.
	\end{enumerate}
	These results are quite standard in the study of multifractal analysis, especially for expanding Markov maps with finite partitions on $[0,1]$. For a summary, see, for example, \cite[\S 2.3]{LiaoSeuretETDS}. Although the content described in that reference differs from our paper, it does indeed lead to items (1)--(3), with $T_\beta$ and $\log \beta$ replaced by the Markov map $T$ and the logarithm of its derivative $\log|T'|$, respectively. To keep the paper at a manageable length, we leave the verification of these details to the interested reader.
\end{rem}
\section{Applications to Gauss map}\label{s:Gauss}
\subsection{Definition and some basic properties}
The Gauss map $G:[0,1)\to[0,1)$ is defined by
\[G(0):=0\qaq G(x)=\frac{1}{x}\ (\textrm{mod}\ 1)\quad\text{for $x\in(0,1)$}.\]
It is well-known that every irrational $x\in(0,1)$ can be written uniquely as an infinite expansion of the form
\begin{equation}\label{eq:cf}
	x=\frac{1}{a_1(x)+\cfrac{1}{a_2(x)+\cfrac{1}{a_3(x)+\cdots}}}=:[a_1(x),a_2(x),a_3(x),\dots],
\end{equation}
where $a_1(x)=\lfloor 1/x\rfloor$ and $a_n(x)=a_1(G^{n-1}x)$ for $n\ge 2$ are called the {\em partial quotients} of $ x $. The $n$th truncation $[a_1(x),\dots,a_n(x)]$, denoted by $p_n(x)/q_n(x)$ is called the $n$th convengent of $x$. With the convention
\[p_{-1}=1,\quad q_{-1}=0,\quad p_0=0\quad\text{and}\quad q_0=1,\]
the convergents $ \{p_n/q_n\}=\{p_n(x)/q_n(x)\} $ of $ x $ can be generated by the recursive formulae:
\begin{equation}\label{eq:recursive}
	p_n=a_np_{n-1}+p_{n-2}\quad\text{and}\quad q_n=a_nq_{n-1}+q_{n-2}\qquad\text{for $ n\ge 1 $}.
\end{equation}
These expressions show that both $p_n$ and $q_n$ are completely determined by the initial segment $\bm{a}_n := (a_1, \dots, a_n) \in \mathbb{N}^n$ of partial quotients. We therefore write
\[p(\bm a_n)=p_n\quad\text{and}\quad q(\bm a_n)=q_n.\]
By the recursive formulae \eqref{eq:recursive}, for any $(a_1,\dots,a_n)\in\N^n$,
\begin{equation}\label{eq:consequence of recursive}
	q(a_1,\dots,a_{n-1})a_n \le q(a_1,\dots,a_n)\le q(a_1,\dots,a_{n-1})(a_n+1).
\end{equation}
\begin{defn}
	For any $ \bm a_n\in\N^n $, we call
	\[I_n(\bm a_n):=\{x\in[0,1):a_k(x)=a_k,1\le k\le n\}\]
	an $ n $th level cylinder.
\end{defn}
Geometrically, these cylinders form a nested partition of the unit interval, refining as $n$ increases. The length of each cylinder decays exponentially with $n$ and can be precisely estimated in terms of the denominators $q_n$ of the convergents:
\begin{lem}[\cite{Khintchinebook,IosifescuKraaikampbook}]\label{l:lengthcylinderGauss}
	Let $\bm a_n\in \N^n$. Then the corresponding $n$th level cylinder satisfies the bounds
	\[q(\bm a_n)^{-2}/2<|I_n(\bm a_n)|\le q(\bm a_n)^{-2},\]
	and moreover,
	\[|I_n(\bm a_n)|\asymp e^{-S_n\log |G'|(x)}\]
	where $x$ belongs to the interior of $I_n(\bm a_n)$.
\end{lem}
The next proposition describes the positions of cylinders of level $n+1$ inside the $n$th level cylinder.

\begin{prop}[\cite{Khintchinebook}] \label{p:distributioncylinder}
	Let $I_n (\bm a_n)$ be an $n$th level cylinder, which is partitioned into sub-cylinders $\{I_{n+1}(\bm a_n, a_{n+1}): a_{n+1} \in \mathbb{N}\}$. When $n$ is odd, these sub-cylinders are positioned from left to right, as $a_{n+1}$ increases from $1$ to $\infty$; when $n$ is even, they are positioned from right to left.
\end{prop}

\subsection{Pressure function}
For any $n\in\N$ and any function $\phi:[0,1)\to\R$, the $n$th variation of $\phi$ is defined by
\[\var_n(\phi):=\sup\{|\phi(x)-\phi(y)|:I_n(x)=I_n(y)\},\]
where $I_n(x)$ denotes the $n$th level cylinder containing $x$. Let $\phi:[0,1)\to\R$ be a real function (not necessarily continuous) with $\var_1(\phi)<\infty$ and $\var_n(\phi)\to 0$ as $n\to\infty$.
The pressure function for Gauss map associated to $\phi$ is defined as
\begin{equation}\label{eq:pressurelimitGauss}
	P(\phi,G):=\lim_{n\to\infty}\frac{1}{n}\log\sum_{\bm a_n\in\N^n}e^{S_n\phi(y)},
\end{equation}
where $y\in I_{n}(\bm a_n)$. The proof of the existence of limit in \eqref{eq:pressurelimitGauss}  can be found in \cite[Proposition 2.4]{LiWangWuXuSTPPLMS}. Clearly, if $\phi$ is either $\log|G'|$ or a continuous function on $[0,1]$, then $\var_1(\phi)<\infty$ and $\var_n(\phi)\to 0$ as $n\to\infty$. Therefore, the pressure for such $\phi$ exists.

 Compared with the $\beta$-transformations, one major difference is that there are infinitely many $n$th level cylinders. As a result, the summation in \eqref{eq:pressurelimitGauss} may be infinite, and hence the pressure function may fail to be continuous with respect to $\phi$. For this reason, instead of providing a comprehensive but technically involved description of the pressure function for the Gauss map, we merely summarize part of the results from \cite{WangWuCFadv,MauldinUrbanskiPLMSIFS,HanusMauldinUrbanskimultifractalAMH} and refer the reader to these references for further details.
\begin{lem}[{\cite[Proposition 3.3]{MauldinUrbanskiPLMSIFS} and \cite[Lemma 2.6]{WangWuCFadv}}]\label{l:continuous function Gauss}
	The function $t\mapsto P(-t\log|G'|,G)$ is strictly decreasing, convex and continuous on $(1/2,\infty)$ and satisfies
	\[\lim_{t\to 1/2^+}P(-t\log|G'|,G)=\infty\qaq P(-\log|G'|,G)=0.\]
\end{lem}
\begin{proof}
	Note that the continued fraction dynamical system can be viewed as an iterated function system:
	\[S=\bigg\{\phi_i(x)=\frac{1}{i+x}:[0,1]\to[0,1]:i\in\N\bigg\}.\]
	It then follows from \cite[Proposition 3.3]{MauldinUrbanskiPLMSIFS} that the pressure function $P(-t\log|G'|,G)$ is strictly decreasing, convex and continuous on $(\theta_S,\infty)$ with
	\[\theta_S:=\inf\{t:P(-t\log|G'|,G)<\infty\}.\]
	For any $n\in \N$, by Lemma \ref{l:lengthcylinderGauss}
	\[\sum_{\bm a_n\in\N^n}e^{-tS_n\log|G'|(y)}\asymp\sum_{\bm a_n\in\N^n}|I_n(\bm a_n)|^t.\]
	Hence, if $t=1$, then by the definition of cylinders, the above sum is approximately $1$, and so $P(-\log|G'|,G)=0$. Now, suppose that $t<1$. Then, by Lemma \ref{l:lengthcylinderGauss} and \eqref{eq:consequence of recursive},
	\[\sum_{\bm a_n\in\N^n}|I_n(\bm a_n)|^t\asymp\sum_{\bm a_n\in\N^n}q(\bm a_n)^{-2t}\le \sum_{\bm a_n\in\N^n}\bigg(\frac{1}{a_1\cdots a_n}\bigg)^{2t}=\bigg(\sum_{k=1}^{\infty}\frac{1}{k^{2t}}\bigg)^n\]
	and similarly,
		\[\sum_{\bm a_n\in\N^n}|I_n(\bm a_n)|^t\gg\bigg(\sum_{k=1}^{\infty}\frac{1}{(k+1)^{2t}}\bigg)^n.\]
		This immediately implies that $P(-t\log|G'|,G)<\infty$ if and only if $t>1/2$. Therefore, $\theta_S=1/2$ and the lemma follows.
\end{proof}

\begin{thm}[{\cite{MauldinUrbanskiPLMSIFS,HanusMauldinUrbanskimultifractalAMH}}]\label{t:GibbsGauss}
	Let $f:[0,1] \to \mathbb{R}$ be a positive Lipschitz continuous function. Then the function $t\mapsto P(-t(f+\log|G'|),G)$ is continuous on $(1/2,+\infty)$. Moreover, the following statements hold:
	\begin{enumerate}[\upshape(1)]
		\item  there exists $t = t(f) \in (1/2, 1)$ such that
		\[P(-t(f+\log|G'|),G)=0.\]
		\item for the function $-t(f + \log |G'|)$, there exists a unique equilibrium state $\nu_t$ satisfying
		\begin{equation}\label{eq:variational principle Gauss}
			0=h_{\nu_t}-t\int (f+\log|G'|)\dif\nu_t,
		\end{equation}
		and such that, for any $\bm a_n \in \mathbb{N}^n$ and any $x \in I_n(\bm a_n)$,
		\[\nu_t(I_n(\bm a_n))\asymp e^{-ntS_n(f+\log|G'|)(x)}.\]
		\item the Hausdorff dimension of $\nu_t$ satisfies
		\begin{equation}\label{eq:dimGibbsGauss}
			\hdim\nu_t=\frac{h_{\nu_t}}{\int\log|G'|\dif\nu_t}.
		\end{equation}
	\end{enumerate}
\end{thm}
\begin{proof}[Sketch of the proof]
(1) Observe that
	\[\begin{split}
		P(-t\log|G'|,G)-t\|f\|_\infty&\le P(-t(f+\log|G'|),G)\\
		&\le P(-t\log|G'|,G)-t\min_{x\in[0,1]}f(x).
	\end{split}\]
	By Lemma \ref{l:continuous function Gauss}, we have
	\[\lim_{t\to 1/2^+}P(-t(f+\log|G'|),G)\ge\lim_{t\to 1/2^+}P(-t\log|G'|,G)-t\|f\|_\infty=\infty\]
	and
	\[P(-(f+\log|G'|),G)\le P(-\log|G'|,G)-\min_{x\in[0,1]}f(x)\le 0.\]
	Since the pressure function is strictly decreasing and continuous in $t$ (see Lemma \ref{l:continuous function Gauss} again), there must exist a unique $t \in (1/2, 1)$ such that
	\[
	P\left(-t(f + \log|G'|), G\right) = 0.
	\]

	(2) The Lipschitz continuity and positivity of $f$ together ensures the existence and uniqueness of the equilibrium state associated with the potential $-t(f + \log|G'|)$; see \cite[Theorem 2.16]{HanusMauldinUrbanskimultifractalAMH}. The variational principle \eqref{eq:variational principle Gauss} follows from their definition of the equilibrium state (see \cite[(2.19)]{HanusMauldinUrbanskimultifractalAMH}) directly. In addition, this equilibrium state satisfies the corresponding Gibbs property; see \cite[(2.16')]{HanusMauldinUrbanskimultifractalAMH}.

	(3) This is a consequence of Birkhoff's ergodic theorem and \cite[Lemma 2.12 (b)]{FanLiaoWangWuETDS}.
\end{proof}

\subsection{Applications}
Note that for any $\bm a_n \in \mathbb{N}^n$, it is known (see, e.g., \cite[Lemma 2.5]{JordanSahlsten16}) that
\[G^n|_{I_n(\bm a_n)}=[0,1)\qaq |(G^n)'(x)|\asymp q(\bm a_n)^2,\]
where $x\in I_n(\bm a_n)$.
The Gibbs property (see Theorem \ref{t:GibbsGauss}) further implies that $\nu_t$ is quasi-self-conformal with respect to the collection of all cylinders.
By a suitable arrangement, it is easy to verify that the $\limsup$ set defined by the collection of cylinders has full Lebesgue measure.
This key observation enables us to apply Theorem~\ref{t:MTP} to the measure $\nu_t$. In light of Theorem~\ref{t:GibbsGauss}, we conclude --- by arguments similar to those used in the proof of the large intersection property of $W(T_\beta,f,\{h_n\})$ (see \S\ref{ss:proofbeta} and Remark \ref{r:applicable to Markov})--- that for any positive continuous function $f:[0,1]\to\mathbb{R}$,
	\[W(G,f, \{h_n\})\in\mathscr G^t([0,1]),\]
where $t$ solves the pressure equation $P(-t(f+\log|G'|),G)=0$.

 We now turn to another class of sets defined in terms of growth conditions on blocks of consecutive partial quotients.
 Recall that for any integer $m \ge 1$ and real number $B > 1$, we define
\[F_m(B):=\{x\in[0,1): a_{n+1}(x)\cdots a_{n+m}(x)\ge B^n\text{ for i.m. $n$}\}.\]
Our goal in the remainder of this subsection is to prove that
\begin{equation}\label{eq:lipconsecutive}
	F_m(B) \in \mathscr{G}^u([0,1]),
\end{equation}
for some $u \in (1/2, 1)$ satisfying \begin{equation}\label{eq:pressure0FmB}
	P(-u \log|G'| - g_m(u) \log B, G) = 0,
\end{equation}
where the function $g_m(u)$ is given by
\[g_m(u)=\frac{u^m(2u-1)}{u^m-(1-u)^m}.\]

The existence of $u$ satisfying \eqref{eq:pressure0FmB} is ensured by the following lemma, which follows from standard properties of the pressure function.
\begin{lem}
	Let $m \ge 1$ be an integer and $B > 1$. There exists $1/2<u<1$ such that
	\[P(-u\log|G'|-g_m(u)\log B,G)=0.\]
\end{lem}
\begin{proof}
	By the definition of the pressure function, we can write
	\[P(-u\log|G'|-g_m(u)\log B,G)=P(-u\log|G'|,G)-g_m(u)\log B.\]
	Let us now consider the two functions of $u$ appearing on the right-hand-side.
	On the one hand, by Lemma \ref{l:continuous function Gauss},
	\[\lim_{u\to 1/2^+}P(-u\log|G'|,G)=\infty,\quad \text{while}\quad P(-\log|G'|,G)=0.\]
	On the other hand, note that $g_m(u)$ is continuous on $(1/2, 1)$, and satisfies
	\[g_m(1/2)\log B=0,\quad\text{and}\quad g_m(1)\log B=\log B>0.\]
	Combining this with the continuity of both functions, it follows that the function
	$u \mapsto P(-u \log|G'|, G) - g_m(u) \log B$
	is continuous and takes values $\infty$ near $u = 1/2$ and negative near $u = 1$. By the intermediate value theorem, there exists some $u \in (1/2, 1)$ such that the equation stated in the lemma equals zero.
\end{proof}
Let $u$ be as in Theorem \ref{t:consecutive}. Denote by $\nu_u$ the Gibbs measure associated to $-u\log|G'|-g_m(u)\log B$. By the dimension formula \eqref{eq:dimGibbsGauss}, we have
\[\hdim\nu_u=\frac{h_{\nu_u}}{\int\log|G'|\dif\nu_u}.\]
By Birkhoff's ergodic theorem, for $\nu_u$-almost all $x$,
\[\lim_{n\to\infty}\frac{1}{n}S_n \log|G'|(x)=\int \log|G'|\dif\nu_u.\]
For any $0<\varepsilon<\int\log|G'|\dif\nu_u/2$, it is not difficult to verify that the set
\begin{equation}\label{eq:ve<}
	\bigcap_{N=1}^\infty\bigcup_{n=N}^\infty \bigcup_{\bm a_n\in\Gamma_n(\nu_u,\varepsilon)}I_n(\bm a_n)
\end{equation}
is of full $\nu_u$-measure, where
\[\Gamma_n(\nu_u,\varepsilon):=\bigg\{\bm a_n\in\N^n:\bigg|\frac{1}{n}S_n \log|G'|(x)-\int \log|G'|\dif\nu_u\bigg|<\varepsilon\text{ for all $x\in I_n(\bm a_n)$}\bigg\}.\]
Consequently, we obtain the inclusion
\[\begin{split}
	F_m(B)\supset\bigcap_{N=1}^\infty\bigcup_{n=N}^\infty \bigcup_{\bm a_n\in\Gamma_n(\nu_u,\varepsilon)}\{x\in I_n(\bm a_n): a_{n+1}(x)\cdots a_{n+m}(x)\ge B^n\}.
\end{split}\]
Let
\[\alpha_i=B^{g_m(u)(1-u)^{i-1}u^{-i}}\quad\text{for $1\le i\le m-1$},\quad\text{and}\quad \alpha_m=\frac{B}{\alpha_1\cdots\alpha_{m-1}}.\]
It can be deduced from the expression of $ g_m(u) $ that the following equalities hold:
\begin{equation}\label{eq:alphai=}
	\alpha_1^u=\alpha_1^{2u-1}\alpha_2^u=\cdots=(\alpha_1\cdots\alpha_{m-1})^{2u-1}\alpha_m^u=B^{g_m(u)}.
\end{equation}

From now on, fix $\bm a_n\in\Gamma_n(\nu_u,\varepsilon)$. We construct an open set inside $\{x\in I_n(\bm a_n): a_{n+1}(x)\cdots a_{n+m}(x)\ge B^n\}$ as follows:
\begin{equation}\label{eq:AGauss}
	A:=\{x\in I_n(\bm a_n): \alpha_i^n\le a_{n+i}(x)\le 2\alpha_i^n \text{ and $a_{n+i}(x)$ is even for $1\le i\le m$} \}.
\end{equation}
Here, we require that $a_{n+i}(x)$ is even to ensure that cylinders of level $n+m$ contained in $A$ are well-separated, in the sense described below.
\begin{lem}\label{l:distance}
	Let $ I_{n+m}(\bm a_n, a_{n+1},\dots, a_{n+m}) $ and $ I_{n+m}(\bm a_n, a_{n+1}',\dots, a_{n+m}') $ be two distinct cylinders contained in $A$. Let $1\le k\le m$ be the smallest integer for which $a_{n+k}\ne a_{n+k}'$. Then, the distance between these two cylinders is at least
	\[\frac{1}{32q(\bm a_n,a_{n+1},\dots,a_{n+k})^2}.\]
\end{lem}
\begin{proof}
	By the distribution properties of cylinders (see Proposition \ref{p:distributioncylinder}) and the fact that by definition both $a_{n+k}$ and $a_{n+k}'$ are even integers, there exists a cylinder
	\[I_{n+k}(\bm a_n, a_{n+1},\dots,a_{n+k-1}, a_{n+k}'')\]
	 with either $a_{n+k}<a_{n+k}''<a_{n+k}'$ or $a_{n+k}'<a_{n+k}''<a_{n+k}$,  lies between the two cylinders stated in the lemma. Therefore, by Lemma \ref{l:lengthcylinderGauss}, \eqref{eq:recursive} and \eqref{eq:AGauss}, they are separated by a distance
	 \[\begin{split}
	 	|I_{n+k}(\bm a_n, a_{n+1},\dots,a_{n+k-1}, a_{n+k}'')|&\ge \frac{1}{2q(\bm a,a_{n+1},\dots,a_{n+k-1},a_{n+k}'')^2}\\
	 	&\ge \frac{1}{32q(\bm a,a_{n+1},\dots,a_{n+k-1},a_{n+k})^2},
	 \end{split}\]
	 which provides the claimed lower bound on the distance between the two cylinders.
\end{proof}
 Define a probability measure $\lambda$ supported on $A$ by
\begin{equation}\label{eq:mu}
	\lambda=\frac{1}{\# A}\sum_{I_{n+m}(\bm a_{n+m})\subset A}\frac{\lm|_{I_{n+m}(\bm a_{n+m})}}{\lm(I_{n+m}(\bm a_{n+m}))},
\end{equation}
where $\lm$ denotes the one-dimensional Lebesgue measure. That is, we assign each $(n+m)$th level cylinder of equal weight. For $1\le k\le m$, the number of descendants of each $(\bm a_n,a_{n+1},\dots, a_{n+k})$ are the same. Therefore, we have
\begin{equation}\label{eq:descendants}
	\lambda(I_{n+k}(\bm a_n,a_{n+1},\dots, a_{n+k}))\asymp \frac{1}{\alpha_1^n\cdots\alpha_m^n}\cdot \alpha_{k+1}^n\cdots\alpha_m^n=\frac{1}{\alpha_1^n\cdots\alpha_k^n}.
\end{equation}
\begin{lem}\label{l:holder}
	Let $ \lambda $ be as above. For any $x\in A$ and $r>0$, we have
	\[\lambda(B(x,r))\ll r^{u-K\varepsilon}q(\bm a_n)^{2(u-K\varepsilon)}B^{ng_m(u)},\]
	where $K=\frac{g_m(u)\log B}{2(\int\log|G'|\dif\nu_u)^2}$.
\end{lem}
\begin{proof}
	Without loss of generality, assume that $ x\in I_{n+m}:=I_{n+m}(\bm a_n,a_{n+1},\dots,a_{n+m})\subset A $. Obviously, if $r$ is relatively large, specifically
	\[r\ge \frac{1}{32q(\bm a_n,a_{n+1}\dots,a_{n+m})^2}\ge \frac{|I_n(\bm a_n)|}{32},\]
	then by Lemma \ref{l:lengthcylinderGauss},
	\[\lambda(B(x,r))\le1\ll \frac{r^{u-K\varepsilon}}{|I_n(\bm a_n)|^{u-K\varepsilon}}\asymp r^{u-K\varepsilon} q(\bm a_n)^{2(u-K\varepsilon)}\le r^{u-K\varepsilon} q(\bm a_n)^{2(u-K\varepsilon)}B^{ng_m(u)}.\]
	Hence, it is sufficient to focus on the case $ r<1/(32\, q(\bm a_n, a_{n+1}, \dots, a_{n+m})^2) $. By Lemma \ref{l:distance}, the cylinders in $A$ are well-separated, allowing us to focus on two distinct cases.

	\noindent {\bf Case 1}: Suppose there exists some $1 \le k \le m$ such that
	\[
	\frac{1}{32\, q(\bm a_n, a_{n+1}, \dots, a_{n+k})^{2}} \le r < \frac{1}{32\, q(\bm a_n, a_{n+1}, \dots, a_{n+k-1})^{2}}.
	\]
	By Lemma \ref{l:distance}, the ball $ B(x,r) $ only intersects one cylinder of level $n+k-1$, namely $ I_{n+k-1}(\bm a_n,a_{n+1}\dots,a_{n+k-1}) $, contained in $A$, but may intersect multiple cylinders of level $n+k$. Define
	\[\begin{split}
		\Delta(x;k)=\{a_{n+k}\in[\alpha_k^n,2\alpha_k^n]:& \text{ $a_{n+k}$ is even and }\\
		&I_{n+k}(\bm a_n,a_{n+1}\dots,a_{n+k-1},a_{n+k})\cap B(x,r)\ne\emptyset\}.
	\end{split}\]
	To estimate $\mu(B(x,r))$, it is essential to bound $\# \Delta(x; k)$ from above. Two natural upper bounds arise:
	\begin{enumerate}[(a)]
		\item From the definition,
		\begin{equation}\label{eq:upper1}
			\# \Delta(x; k) \ll \alpha_k^n.
		\end{equation}

		\item  From the well-separation property (see Lemma \ref{l:distance}), cylinders of level $n+k$ in $A$ are spaced by at least $1/(32\, q(\bm a_n, a_{n+1}, \dots, a_{n+k})^2)$. Thus,
		\begin{equation}\label{eq:upper2}
			\# \Delta(x; k) \ll r \, q(\bm a_n, a_{n+1}, \dots, a_{n+k})^2 \stackrel{\eqref{eq:consequence of recursive}}{\asymp} r \, q(\bm a_n)^2 a_{n+1}^2 \cdots a_{n+k}^2.
		\end{equation}
	\end{enumerate}
	Combining \eqref{eq:upper1} and \eqref{eq:upper2} and using the inequality $\min\{a,b\}\le a^{1-u}b^u$, we get
	\begin{align}
		\#\Delta(x;k)&\ll\min\{\alpha_k^n,rq(\bm a_n)^2a_{n+1}^2\cdots a_{n+k}^2\}\notag\\
		&\ll \alpha_k^{n(1-u)} \cdot \left( r q(\bm a_n)^2 a_{n+1}^2 \cdots a_{n+k}^2 \right)^u\notag\\
		&\ll r^uq(\bm a_n)^{2u}\alpha_1^{2nu}\cdots \alpha_{k-1}^{2nu}\alpha_k^{n(1+u)},\label{eq:Delta(x,k)}
	\end{align}
	where we use $a_{n+i}\asymp \alpha_i^n$ for $1\le i\le m$ in the last inequality.
	Hence,
	\[\begin{split}
		\lambda(B(x,r))&\stackrel{\eqref{eq:descendants}}{\ll} \#\Delta(x;k)\cdot\frac{1}{\alpha_1^n\cdots\alpha_k^n}\\
		&\stackrel{\eqref{eq:Delta(x,k)}}{\ll}r^uq(\bm a_n)^{2u}\alpha_1^{n(2u-1)}\cdots \alpha_{k-1}^{n(2u-1)}\alpha_k^{nu}\\
		&\stackrel{\eqref{eq:alphai=}}{=} r^uq(\bm a_n)^{2u}B^{ng_m(u)}\ll r^{u-K\varepsilon}q(\bm a_n)^{2(u-K\varepsilon)}B^{ng_m(u)},
	\end{split}\]
	where the last inequality follows from $r\ll q(\bm a_n)$.

	\noindent Case 2: If \[
	r \le \frac{1}{32\, q(\bm a_n, a_{n+1}, \dots, a_{n+m})^2},
	\]
	then $B(x,r)$ intersects only one cylinder of level $n+m$, namely $I_{n+m}$, contained in $A$. It follows that
	\[\begin{split}
		\frac{\lm|_{I_{n+m}}(B(x,r))}{\lm(I_{n+m})}&\le\frac{2r}{\lm(I_{n+m})}\stackrel{\text{Lemma \ref{l:lengthcylinderGauss}}}{\asymp} rq(\bm a_n,a_{n+1},\dots,a_{n+m})^2\\
		&\ll r^uq(\bm a_n,a_{n+1},\dots,a_{n+m})^{2u}\stackrel{\eqref{eq:consequence of recursive}}{\asymp} r^uq(\bm a_n)^{2u}\alpha_1^{2nu}\cdots\alpha_m^{2nu}\\
		&\le r^uq(\bm a_n)^{2u}\alpha_1^{2nu}\cdots\alpha_{m-1}^{2nu}\alpha_m^{n(u+1)},
	\end{split}\]
	where we use $1/2<u<1$ in the last inequality.
	This together with the definition of $\lambda$ gives
	\[\begin{split}
		\lambda(B(x,r))&=\frac{1}{\#A}\cdot 	\frac{\lm|_{I_{n+m}}(B(x,r))}{\lm(I_{n+m})}	\ll r^uq(\bm a_n)^{2u}\alpha_1^{n(2u-1)}\cdots \alpha_{m-1}^{n(2u-1)}\alpha_m^{nu}\\
		&\stackrel{\eqref{eq:alphai=}}{=}r^uq(\bm a_n)^{2u}B^{ng_m(u)}\ll r^{u-K\varepsilon}q(\bm a_n)^{2(u-K\varepsilon)}B^{ng_m(u)}.\qedhere
	\end{split}\]
\end{proof}
We are now in a position to prove Theorem~\ref{t:consecutive} (see also \eqref{eq:lipconsecutive}), using the measure $\lambda$ constructed earlier and the mass distribution principle.
\begin{proof}[Proof of Theorem \ref{t:consecutive}]
		Let $\bm a_n \in \Gamma_n(\nu_u, \varepsilon)$. By Lemma~\ref{l:holder} and the mass distribution principle, we obtain the following lower bound for the $(u - K\varepsilon)$-Hausdorff content of the set $A$:
	\begin{equation}\label{eq:hcA}
		\hc^{u-K\varepsilon}(A)\gg q(\bm a_n)^{-2(u-K\varepsilon)}B^{-ng_m(u)}\asymp e^{-(u-K\varepsilon)S_n\log|G'|(x)-ng_m(u)\log B},
	\end{equation}
	where $x \in I_n(\bm a_n)$ is any point in the cylinder. To proceed, we analyze the exponent on the right-hand-side of~\eqref{eq:hcA}. By the definition of $\Gamma_n(\nu_u, \varepsilon)$, we know that for any $x \in I_n(\bm a_n)$,
	\[\bigg|\frac{1}{n}S_n \log|G'|(x)-\int \log|G'|\dif\nu_u\bigg|<\varepsilon~\Longrightarrow~ S_n \log|G'|(x)\ge n\bigg(\int \log|G'|\dif\nu_u-\varepsilon\bigg).\]
	Substituting this into the exponent, we obtain:
	\begin{align}
		&(u-K\varepsilon)S_n\log|G'|(x)+ng_m(u)\log B\notag\\
		=&S_n\log|G'|(x)\bigg(u-K\varepsilon+\frac{ng_m(u)\log B}{S_n\log|G'|(x)}\bigg)\notag\\
		\le& S_n\log|G'|(x)\bigg(u-K\varepsilon+\frac{ng_m(u)\log B}{n(\int \log|G'|\dif\nu_u-\varepsilon)}\bigg)\notag.
	\end{align}
	Recall that $0<\varepsilon<\int \log|G'|\dif\nu_u/2$ (see the line before \eqref{eq:ve<}) and $K=\frac{g_m(u)\log B}{2(\int\log|G'|\dif\nu_u)^2}$. It follows that
	\[\begin{split}
		&\frac{ng_m(u)\log B}{n(\int \log|G'|\dif\nu_u-\varepsilon)}-K\varepsilon\\
		=& \frac{g_m(u)\log B}{\int \log|G'|\dif\nu_u}\cdot \frac{\int \log|G'|\dif\nu_u}{\int \log|G'|\dif\nu_u-\varepsilon}-\frac{g_m(u)\log B}{2(\int\log|G'|\dif\nu_u)^2}\cdot\varepsilon\\
		=&\frac{g_m(u)\log B}{\int \log|G'|\dif\nu_u}+\frac{g_m(u)\log B}{\int \log|G'|\dif\nu_u}\cdot \frac{\varepsilon}{\int \log|G'|\dif\nu_u-\varepsilon}-\frac{g_m(u)\log B}{2(\int\log|G'|\dif\nu_u)^2}\cdot\varepsilon\\
		\le& \frac{g_m(u)\log B}{\int \log|G'|\dif\nu_u}.
	\end{split}\]
	Therefore,
	\begin{equation}\label{eq:exponent}
	(u-K\varepsilon)S_n\log|G'|(x)+ng_m(u)\log B\le S_n\log|G'|(x)\bigg(u+\frac{g_m(u)\log B}{\int \log|G'|\dif\nu_u}\bigg).
	\end{equation}
		By the variational principle (see Theorem \ref{t:GibbsGauss} (2)), we have
	\[\begin{split}
		&0=h_{\nu_u}-\bigg(u\int \log|G'|\dif\nu_u+g_m(u)\log B\bigg)\\
		\Longrightarrow\quad& u+\frac{g_m(u)\log B}{\int \log|G'|\dif\nu_u}=\frac{h_{\nu_u}}{\int \log|G'|\dif\nu_u}=\hdim\nu_u.
	\end{split}\]
	Substituting this identity into~\eqref{eq:exponent} and then into~\eqref{eq:hcA}, we conclude that
	\[\hc^{u-K\varepsilon}(A)\gg e^{-S_n\log|G'|(x)\cdot \hdim\nu_u}\stackrel{\text{Lemma \ref{l:lengthcylinderGauss}}}{\asymp} |I_n(\bm a_n)|^{\hdim\nu_u}.\]
		Finally, observe that by construction,
	\[
	A \subset \left\{ x \in I_n(\bm a_n) : a_{n+1}(x) \cdots a_{n+m}(x) \ge B^n \right\}\qaq \bm a_n \in \Gamma_n(\nu_u, \varepsilon),
	\]
	so we have the lower bound
	\[
	\hc^{u - K\varepsilon}\left( \left\{ x \in I_n(\bm a_n) : a_{n+1}(x) \cdots a_{n+m}(x) \ge B^n \right\} \right) \ge |I_n(\bm a_n)|^{\hdim \nu_u}.
	\]
	Since this holds for all $\bm a_n \in \Gamma_n(\nu_u, \varepsilon)$, by \eqref{eq:ve<} and Theorem~\ref{t:MTP}, it follows that
	\[
	F_m(B) \in \mathscr{G}^{u - K\varepsilon}([0,1]).
	\]
	Since $\varepsilon > 0$ is arbitrary, we conclude that
	\[
	F_m(B) \in \mathscr{G}^{u}([0,1]),
	\]
	which completes the proof.
\end{proof}
\subsection*{Acknowledgments}
 The authors would like to thank the referee very much for care-
ful reading of the manuscript, pointing out some mistakes and giving many beneficial
suggestions. Y. He was supported by the NSFC (No. 12401108) and partially by a grant from the Guangdong Provincial Department of Education (2025KCXTD013).
%The author is very grateful to the anonymous referees for their patience and careful reading, which has helped improve the quality of the manuscript.

\bibliographystyle{abbrv}

\bibliography{bibliography}

\begin{thebibliography}{10}

\bibitem{AllenBakerMTPselect}
D.~Allen and S.~Baker.
\newblock A general mass transference principle.
\newblock {\em Selecta Math. (N.S.)}, 25(3):Paper No. 39, 38, 2019.

\bibitem{AllenBaranySTPMathematika}
D.~Allen and B.~B{\'a}r{\'a}ny.
\newblock On the {H}ausdorff measure of shrinking target sets on self-conformal
  sets.
\newblock {\em Mathematika}, 67(4):807--839, 2021.

\bibitem{AllenBeresnevichhyperplaneCompos}
D.~Allen and V.~Beresnevich.
\newblock A mass transference principle for systems of linear forms and its
  applications.
\newblock {\em Compos. Math.}, 154(5):1014--1047, 2018.

\bibitem{BakerKoivusaloSTPadv}
S.~Baker and H.~Koivusalo.
\newblock Quantitative recurrence and the shrinking target problem for
  overlapping iterated function systems.
\newblock {\em Advances in Mathematics}, 442:109538, 2024.

\bibitem{BaranyRamsSTPPLMS}
B.~B\'ar\'any and M.~Rams.
\newblock Shrinking targets on {B}edford-{M}c{M}ullen carpets.
\newblock {\em Proc. Lond. Math. Soc. (3)}, 117(5):951--995, 2018.

\bibitem{BarralSeuretMTP}
J.~Barral and S.~Seuret.
\newblock Heterogeneous ubiquitous systems in {$\Bbb R^d$} and {H}ausdorff
  dimension.
\newblock {\em Bull. Braz. Math. Soc. (N.S.)}, 38(3):467--515, 2007.

\bibitem{BeresnevichVelaniMTPann}
V.~Beresnevich and S.~Velani.
\newblock A mass transference principle and the {D}uffin-{S}chaeffer conjecture
  for {H}ausdorff measures.
\newblock {\em Ann. of Math. (2)}, 164(3):971--992, 2006.

\bibitem{BesicovitchJLMSDiophantine}
A.~S. Besicovitch.
\newblock Sets of {F}ractional {D}imensions ({IV}): {O}n {R}ational
  {A}pproximation to {R}eal {N}umbers.
\newblock {\em J. London Math. Soc.}, 9(2):126--131, 1934.

\bibitem{BishopPeresbook}
C.~J. Bishop and Y.~Peres.
\newblock {\em Fractals in probability and analysis}, volume 162 of {\em
  Cambridge Studies in Advanced Mathematics}.
\newblock Cambridge University Press, Cambridge, 2017.

\bibitem{BugeaudWangSTPJFG}
Y.~Bugeaud and B.-W. Wang.
\newblock Distribution of full cylinders and the {D}iophantine properties of
  the orbits in $\beta$-expansions.
\newblock {\em Journal of Fractal Geometry}, 1(2):221--241, 2014.

\bibitem{CoonsHussainWangSTPbeta}
M.~Coons, M.~Hussain, and B.-W. Wang.
\newblock A dichotomy law for the diophantine properties in {$\beta$}-dynamical
  systems.
\newblock {\em Mathematika}, 62(3):884--897, 2016.

\bibitem{DaviaudMTPJMAA}
E.~Daviaud.
\newblock A dimensional mass transference principle from ball to rectangles for
  projections of {G}ibbs measures and applications.
\newblock {\em J. Math. Anal. Appl.}, 538(1):Paper No. 128386, 28, 2024.

\bibitem{DaviaudMTPadv}
E.~Daviaud.
\newblock A dimensional mass transference principle for {B}orel probability
  measures and applications.
\newblock {\em Adv. Math.}, 474:Paper No. 110304, 47, 2025.

\bibitem{DaviaudSTPETDS}
E.~Daviaud.
\newblock Dynamical {D}iophantine approximation and shrinking targets for
  {$C^1$} weakly conformal {IFS}s with overlaps.
\newblock {\em Ergodic Theory Dynam. Systems}, 45(6):1777--1826, 2025.

\bibitem{ErikssonMTPadv}
S.~Eriksson-Bique.
\newblock A new {H}ausdorff content bound for limsup sets.
\newblock {\em Adv. Math.}, 445:Paper No. 109638, 52, 2024.

\bibitem{FalconerLIPintroduce}
K.~Falconer.
\newblock Sets with large intersection properties.
\newblock {\em J. London Math. Soc. (2)}, 49:267--280, 1994.

\bibitem{Falconer_book}
K.~J. Falconer.
\newblock {\em Fractal geometry}.
\newblock John Wiley \& Sons, Ltd., Chichester, 1990.
\newblock Mathematical foundations and applications.

\bibitem{FanLiaoWangWuETDS}
A.-H. Fan, L.-M. Liao, B.-W. Wang, and J.~Wu.
\newblock On {K}hintchine exponents and {L}yapunov exponents of continued
  fractions.
\newblock {\em Ergodic Theory Dynam. Systems}, 29(1):73--109, 2009.

\bibitem{FanWangbetaNon}
A.-H. Fan and B.-W. Wang.
\newblock On the lengths of basic intervals in beta expansions.
\newblock {\em Nonlinearity}, 25(5):1329--1343, 2012.

\bibitem{HanusMauldinUrbanskimultifractalAMH}
P.~Hanus, R.~D. Mauldin, and M.~Urba\'nski.
\newblock Thermodynamic formalism and multifractal analysis of conformal
  infinite iterated function systems.
\newblock {\em Acta Math. Hungar.}, 96(1-2):27--98, 2002.

\bibitem{HeSTPETDS}
Y.~He.
\newblock Shrinking parallelepiped targets for {$\beta$}-dynamical systems.
\newblock {\em Ergodic Theory Dynam. Systems}, 45(6):1827--1842, 2025.

\bibitem{HeMTPadv}
Y.~He.
\newblock A unified approach to mass transference principle and large
  intersection property.
\newblock {\em Adv. Math.}, 471:Paper No. 110267, 51, 2025.

\bibitem{HillVelaniSTPintroduce}
R.~Hill and S.~L. Velani.
\newblock The ergodic theory of shrinking targets.
\newblock {\em Inventiones mathematicae}, 119(1):175--198, 1995.

\bibitem{HuangWuXuconsecutiveISM}
L.~Huang, J.~Wu, and J.~Xu.
\newblock Metric properties of the product of consecutive partial quotients in
  continued fractions.
\newblock {\em Israel J. Math.}, 238(2):901--943, 2020.

\bibitem{HussainLiSimmonsWangBCETDS}
M.~Hussain, B.~Li, D.~Simmons, and B.~Wang.
\newblock Dynamical {B}orel-{C}antelli lemma for recurrence theory.
\newblock {\em Ergodic Theory Dynam. Systems}, 42(6):1994--2008, 2022.

\bibitem{IosifescuKraaikampbook}
M.~Iosifescu and C.~Kraaikamp.
\newblock {\em Metrical theory of continued fractions}, volume 547 of {\em
  Mathematics and its Applications}.
\newblock Kluwer Academic Publishers, Dordrecht, 2002.

\bibitem{Jarnikdimension}
V.~Jarn\'ik.
\newblock {D}iophantische {A}pproximationen und {H}ausdorffsches {M}ass.
\newblock {\em Mat. Sb.}, 36:371--382, 1929.

\bibitem{JordanSahlsten16}
T.~Jordan and T.~Sahlsten.
\newblock Fourier transforms of {G}ibbs measures for the {G}auss map.
\newblock {\em Math. Ann.}, 364(3-4):983--1023, 2016.

\bibitem{Khintchinebook}
A.~Y. Khinchin.
\newblock {\em Continued fractions}.
\newblock University of Chicago Press, Chicago, Ill.-London, 1964.

\bibitem{KleinbockZhengBCNonlinearity}
D.~Kleinbock and J.~Zheng.
\newblock Dynamical {B}orel-{C}antelli lemma for recurrence under {L}ipschitz
  twists.
\newblock {\em Nonlinearity}, 36(2):1434--1460, 2023.

\bibitem{KoivusaloRamsMTPIMRN}
H.~Koivusalo and M.~Rams.
\newblock Mass transference principle: from balls to arbitrary shapes.
\newblock {\em Int. Math. Res. Not. IMRN}, (8):6315--6330, 2021.

\bibitem{LiLiaoVelaniWangZorinMTPadv}
B.~Li, L.~Liao, S.~Velani, B.~Wang, and E.~Zorin.
\newblock Diophantine approximation and the mass transference principle:
  incorporating the unbounded setup.
\newblock {\em Adv. Math.}, 470:Paper No. 110248, 61, 2025.

\bibitem{LiLiaoVelaniZorinSTPadv}
B.~Li, L.~Liao, S.~Velani, and E.~Zorin.
\newblock The shrinking target problem for matrix transformations of tori:
  Revisiting the standard problem.
\newblock {\em Advances in Mathematics}, 421:108994, 2023.

\bibitem{LiWangWuXuSTPPLMS}
B.~Li, B.-W. Wang, J.~Wu, and J.~Xu.
\newblock The shrinking target problem in the dynamical system of continued
  fractions.
\newblock {\em Proceedings of the London Mathematical Society},
  108(1):159--186, 2014.

\bibitem{LiaoSeuretETDS}
L.~Liao and S.~Seuret.
\newblock Diophantine approximation by orbits of expanding {M}arkov maps.
\newblock {\em Ergodic Theory Dynam. Systems}, 33(2):585--608, 2013.

\bibitem{Mattilageometry1999}
P.~Mattila.
\newblock {\em Geometry of sets and measures in Euclidean spaces: fractals and
  rectifiability}.
\newblock Number~44. Cambridge university press, 1999.

\bibitem{MauldinUrbanskiPLMSIFS}
R.~D. Mauldin and M.~Urba\'nski.
\newblock Dimensions and measures in infinite iterated function systems.
\newblock {\em Proc. London Math. Soc. (3)}, 73(1):105--154, 1996.

\bibitem{PerssonMTPLIPreal}
T.~Persson.
\newblock A mass transference principle and sets with large intersections.
\newblock {\em Real Anal. Exchange}, 47(1):191--205, 2022.

\bibitem{Philipp1967}
W.~Philipp.
\newblock Some metrical theorems in number theory.
\newblock {\em Pacific Journal of Mathematics}, 20(1), 1967.

\bibitem{RenyibetaAMH}
A.~R\'enyi.
\newblock Representations for real numbers and their ergodic properties.
\newblock {\em Acta Math. Acad. Sci. Hungar.}, 8:477--493, 1957.

\bibitem{TanWangRecurbetaadv}
B.~Tan and B.-W. Wang.
\newblock Quantitative recurrence properties for beta-dynamical system.
\newblock {\em Advances in Mathematics}, 228(4):2071--2097, 2011.

\bibitem{WaltersbetaGibbs}
P.~Walters.
\newblock Equilibrium states for {$\beta $}-transformations and related
  transformations.
\newblock {\em Math. Z.}, 159(1):65--88, 1978.

\bibitem{WangWusurvey}
B.~Wang and J.~Wu.
\newblock A survey on the dimension theory in dynamical {D}iophantine
  approximation.
\newblock In {\em Recent developments in fractals and related fields
  \uppercase\expandafter{\romannumeral2}}, Trends Math., pages 261--294.
  Birkh\"{a}user/Springer, Cham, 2017.

\bibitem{WangWuMTPrectangleMA}
B.~Wang and J.~Wu.
\newblock Mass transference principle from rectangles to rectangles in
  {D}iophantine approximation.
\newblock {\em Math. Ann.}, 381(1-2):243--317, 2021.

\bibitem{WangWuCFadv}
B.-W. Wang and J.~Wu.
\newblock Hausdorff dimension of certain sets arising in continued fraction
  expansions.
\newblock {\em Adv. Math.}, 218(5):1319--1339, 2008.

\bibitem{WangZhangDMPMathZ}
B.-W. Wang and G.-H. Zhang.
\newblock A dynamical dimension transference principle for dynamical
  diophantine approximation.
\newblock {\em Math. Z.}, 298(1-2):161--191, 2021.

\bibitem{WangSTPJMAA}
W.~Wang.
\newblock Modified shrinking target problem in beta dynamical systems.
\newblock {\em J. Math. Anal. Appl.}, 468(1):423--435, 2018.

\bibitem{ZhongMTPJMAA}
W.~Zhong.
\newblock Mass transference principle: from balls to arbitrary shapes: measure
  theory.
\newblock {\em J. Math. Anal. Appl.}, 495(1):Paper No. 124691, 23, 2021.

\end{thebibliography}

\end{document}